\documentclass[11pt,a4paper,twoside]{scrartcl}
\usepackage[english]{babel}

\usepackage{amsfonts}
\usepackage{amsmath}
\usepackage{amssymb}
\usepackage{amsthm}

\usepackage{graphicx}
\graphicspath{{pics/}}

\usepackage{enumerate}
\usepackage{verbatim}

\usepackage{
	float,		% Positionierung [H] unterbindet floating
	subcaption, % f++r subfigure und subtable
	xfrac,		% schr+\~{n}ger Bruchstrich
	colortbl,
	mathtools,  % zur Benutzung von \coloneqq (:=)
	pifont,  	% zur Benutzung von \ding{55}
	mathrsfs	% zur Benutzung von \mathscr
}

\usepackage[symbol]{footmisc}
\usepackage{enumitem} % to change enumerate style

\usepackage{hyperref}
\usepackage{todonotes}
\usepackage{pgfplots}
\pgfplotsset{compat=1.8}
\pgfplotsset{
	FTerrstyle/.style={
		width=\textwidth, height=5.6cm,
		grid=both,
		grid style={gray!18, line width=0.2pt},
		tick label style={font=\footnotesize},
		label style={font=\small},
		legend cell align=left,
		legend style={font=\footnotesize, fill=white, fill opacity=0.85,
			text opacity=1, draw=gray!40, inner sep=2pt},
		every axis plot/.append style={line width=0.9pt},
		mark size=1.3pt,
	},
	FTest/.style ={red!80!black, dashed, mark=none},
	FTsim/.style ={blue!75!black, mark=*, mark options={fill=blue!55!white}},
}
\usepackage{tikz}
\usepackage{tikz-cd}
\newtheorem{theorem}{Theorem}[section]
\newtheorem{lemma}[theorem]{Lemma}
\newtheorem{remark}[theorem]{Remark}

\newtheorem{example}[theorem]{Example}
\newtheorem{corollary}[theorem]{Corollary}

\newtheorem{alg}[theorem]{Algorithm}

\makeatletter
\def\imod#1{\allowbreak\mkern10mu({\operator@font mod}\,\,#1)}
\makeatother

\makeatletter
\let\c@algorithm\c@theorem  % algorithms numbered like definitions, ...
\makeatother

\newenvironment{algorithm}[1]{\goodbreak~\begin{alg}[#1]~\vspace{-9pt}~\\
		\rule{\linewidth}{0.5pt}~\\}{\vspace{-9pt}~\\
		\rule{\linewidth}{0.5pt}~\end{alg}}

\numberwithin{equation}{section}
\numberwithin{table}{section}
\numberwithin{figure}{section}

\usepackage[normalem]{ulem}
\newcommand{\new}[1]{#1}
\newcommand{\del}[1]{}
\newcommand{\FFT}{\textup{FFT}}
\newcommand{\NFFT}{\textup{NFFT}}
\newcommand{\R}{\mathbb R}

\newcommand{\Z}{\mathbb Z}

\renewcommand{\qedsymbol}{\rule{1.5ex}{1.5ex}}
\long\def\symbolfootnote[#1]#2{\begingroup%
	\def\thefootnote{\fnsymbol{footnote}}\footnote[#1]{#2}\endgroup}

\allowdisplaybreaks
\title{Computation of the Fourier transform for a continuous integrable function via NFFT}

\date{\today}
\author{Daniel Potts\footnotemark[1] \and Manfred Tasche\footnotemark[3]}

\hypersetup{pdfauthor={Daniel Potts, Manfred Tasche},
	pdftitle={Computation of the Fourier transform via  NFFT},
	pdfsubject={},
	pdfcreator = {pdflatex},
	plainpages=false,
	pdfstartview=FitH,
	pdfview=FitH,
	pdfpagemode=UseOutlines,
	bookmarksnumbered=true,
	bookmarksopen=false,
	bookmarksopenlevel=0,
	colorlinks=true,
	linkcolor=black,
	citecolor=black,
	urlcolor=black}

\begin{document}
	
	\maketitle
	
\begin{abstract}
	We investigate the approximation of continuous Fourier transforms by
	trigonometric sampling polynomials and their efficient evaluation by the
	nonequispaced fast Fourier transform (NFFT). While the NFFT is traditionally
	used for the evaluation of trigonometric polynomials, we show that it can also
	serve as an effective computational tool for the approximation of Fourier
	transform values.
	Building on ideas of M.~Ehler, K.~Gr\"{o}chenig, and A.~Klotz \cite{EhGrKl24}, we derive explicit 
	$\ell_\infty$ uniform error bounds between the Fourier transform and suitable sampling
	polynomials. The resulting estimates quantify the influence of the sampling
	width and truncation parameter and provide rigorous accuracy guarantees on
	entire frequency intervals. In contrast to previous analyses focusing mainly on
	discrete or $L_2$-type errors, our results yield uniform approximation bounds
	that are directly relevant for practical computations.
	\new{Moreover, we prove a matching lower bound which shows that the derived
	rate in the truncation parameter is sharp.}

	The derived theory leads to a simple algorithmic framework: first approximate
	the Fourier transform by a trigonometric sampling polynomial and then evaluate this
	polynomial efficiently by the NFFT. Numerical experiments confirm the
	theoretical convergence rates and demonstrate that accurate approximations of
	continuous Fourier transforms can be obtained with moderate computational
	effort.

		\smallskip
		
		\noindent AMS \emph{Subject Classifications}: \text{42A38, 65T50, 94A20}
		%	%42A10 % Trigonometric approximation
		%		%	%42B05 % Fourier series and coefficients
		%		%	%42C15 % General harmonic expansions, frames
		%		%	%65-XX % Numerical Analysis
		%		%	65D05, % Numerical approximation - Interpolation
		%		%	65D30 % Numerical integration
		%		%	65D32, % Quadrature and cubature formulas
		%		65Txx, % Numerical methods in Fourier Analysis
		%		65T50, % Discrete and fast Fourier transforms
		%%		65F05. % Direct numerical methods for linear systems and matrix inversion
		%		%	65F10, % Iterative numerical methods for linear systems
		%		%	%65T40, % Trigonometric approximation and interpolation
		%		94A20 % Sampling theory
		%		%	94A24. % Coding theorems (Shannon theory)		
	\end{abstract}
	
\footnotetext[1]{potts@math.tu-chemnitz.de, Chemnitz University of
		Technology, Faculty of Mathematics, D--09107 Chemnitz, Germany}
	\footnotetext[3]{manfred.tasche@math.uni-rostock.de,  University of Rostock, Institute of Mathematics, D--18051 Rostock, Germany}

	%===============================================================================

	\section{Introduction}
	
The nonequispaced fast Fourier transform (NFFT) is a well-established tool for
the efficient evaluation of trigonometric polynomials at arbitrary nodes
\cite{duro93,bey95,st97,PlPoStTa23,GrLe04,FINUFFTpaper}.
Traditionally, the NFFT is employed once the Fourier coefficients of a
trigonometric polynomial are already known. In many applications, however, the
primary quantity of interest is not the trigonometric polynomial itself but the
continuous Fourier transform
	\begin{equation}
\label{eq:hatf}
{\hat f}(v) = \int_{\mathbb{R}} f(x)\,{\mathrm{e}}^{-2\pi{\mathrm{i}}\,xv}\,{\mathrm{d}}x\,,
\quad v \in \mathbb{R}\,,
\end{equation}

This naturally raises the question whether the NFFT can also be used as a
numerical method for computing values of the Fourier transform.

The present paper provides a positive answer to this question. Building upon
ideas developed by M.~Ehler, K.~Gr\"{o}chenig, and A.~Klotz \cite{EhGrKl24}, we interpret
trigonometric sampling polynomials as approximations of the Fourier transform
and derive explicit uniform error estimates. These results show that the NFFT
can be employed not only for the evaluation of trigonometric polynomials but
also for the efficient computation of values of the continuous Fourier
transform.
Consequently, the NFFT becomes more than a fast evaluation scheme for
trigonometric polynomials. Combined with the approximation results developed in
this paper, it yields a practical and theoretically justified method for the
computation of continuous Fourier transforms.	

More precisely, we consider the computation of the Fourier transform of $\hat f$, see \eqref{eq:hatf},
	under certain decay conditions on $f$ and $\hat{f}$.
	We restrict $f$ to the spatial domain $[-L/2,\, L/2]$ and $\hat{f}$ to the
	frequency domain $[-P/2,\, P/2]$, where $L \in 2\mathbb{N}$ and $P \in 2\mathbb{N}$
	are sufficiently large positive integers. We sample $f$ at the equispaced points
	$j/P$, $j \in [LP]$, with spatial step size $1/P$, where
	$$
	[LP] := \bigl\{j \in \mathbb{Z} : -\tfrac{LP}{2} < j \leq \tfrac{LP}{2}\bigr\}
	$$
	denotes the centered index set. The $P$-\emph{periodic trigonometric sampling
		polynomial of} $f$ \emph{with degree} $LP/2$ is defined by
	\begin{equation}\label{eq:samplingpoly}
		s_{P,L}(v) := \frac{1}{P}\sum_{n \in [LP]} f\!\left(\frac{n}{P}\right)
		{\mathrm{e}}^{-2\pi{\mathrm{i}}\,nv/P}\,, \quad v \in \mathbb{R}\,,
	\end{equation}
which approximates ${\hat f}(v)$ on the compact frequency interval $[-P/2,\,P/2]$.
	We determine approximate values of ${\hat f}\!\left(\frac{k}{L}\right)$,
	$k \in [LP]$, by evaluating the trigonometric sampling polynomial $s_{P,L}(v)$ at $v = k/L$, giving
	\begin{equation}
		\label{eq:gell}
		s_{P,L}\!\left(\tfrac{k}{L}\right)
		= \frac{1}{P} \sum_{j \in [LP]} f\!\left(\frac{j}{P}\right)
		{\mathrm{e}}^{-2\pi{\mathrm{i}}\,jk/(LP)}\,, \quad k \in [LP]\,.
	\end{equation}
	This standard approach to computing the Fourier transform is well documented in
	the literature (see, e.g., \cite[pp.~17--23]{BH95},
	\cite[pp.~507--508]{PlPoStTa23}, or \cite{EhGrKl24}). Note that the values \eqref{eq:gell}
	can be computed efficiently via the fast Fourier transform (FFT), provided
	$L$ and $P$ are products of powers of the small primes, see \cite{FJ05}. More generally, the trigonometric sampling polynomial $s_{P,L}(v)$ can be
	evaluated at arbitrary points $v$ with $|v| \leq P/2$ by means of the
	nonequispaced fast Fourier transform (NFFT). Recall that the NFFT is a fast approximate
	algorithm to evaluate a 1-periodic trigonometric polynomial
	\begin{equation}
		\label{eq:trig_poly}
		\sum_{k \in [M]} c_{k}\,{\mathrm{e}}^{2\pi{\mathrm{i}}\,kx},
		\quad x \in \mathbb{R},
	\end{equation}
	with given coefficients $c_k \in \mathbb{C}$, $k \in [M]$, at
	given nonequispaced points $x_j \in [0,\,1]$, $j = 1,\dots,N$,
	where  $M \in 2\mathbb{N}$.
	The fundamental idea of the NFFT is to exploit the FFT to efficiently switch
	between the spatial and frequency domains, thereby enabling a computationally
	efficient approximation at arbitrary nonuniform points.
	Throughout this paper, we assume that $L,\, P \in 2\mathbb{N}$ are sufficiently large.
	
	In \cite{EhGrKl24}, M.~Ehler, K.~Gr\"ochenig, and A.~Klotz established error estimates
	for this computation of the values ${\hat f}\!\left(\tfrac{k}{L}\right)$ for $k \in [L P]$. Specifically, they studied how the scaled Euclidean
	approximation error
	$$
	E_{P,L}^{\ast}(f) := \frac{1}{\sqrt{L}}\left(\sum_{k \in [LP]}
	\left|{\hat f}\!\left(\tfrac{k}{L}\right) - s_{P,L}\!\left(\tfrac{k}{L}\right)\right|^2\right)^{1/2}
	$$
	depends on $P$ and $L$, given known decay behavior of $f \in C(\mathbb{R})$ and
	$\hat{f} \in C(\mathbb{R})$. These estimates are formulated in terms of norms of Wiener
	amalgam spaces \cite{Fe92a}.
	
	In this paper, we address the analogous question for the scaled maximum
	approximation error
	\begin{equation}
		\label{eq:uniform}
		M_{P,L}(f) := \frac{1}{\sqrt L}\,\max_{|v| \leq P/2} \bigl|{\hat f}(v) - s_{P,L}(v)\bigr|\,,
	\end{equation}
	again under decay conditions on $f$ and $\hat{f}$. In particular, we study the
	discretized version
	\begin{equation}
\label{eq:discreteuniform}
M_{P,L}^{\ast}(f) := \frac{1}{\sqrt L}\,
	\max_{k \in [LP]} \left|{\hat f}\!\left(\tfrac{k}{L}\right)
	- s_{P,L}\!\left(\tfrac{k}{L}\right)\right|\,.
	\end{equation}
Obviously, it holds
$M_{P,L}^{\ast}(f) \leq M_{P,L}(f)$.
Therefore we are interested in practicable, explicit upper bounds of $M_{P,L}(f)$. We discuss the following questions:
%\begin{enumerate}
%\item How should $L$, $P \in 2 \mathbb N$ be chosen depending on the decay behavior of $f$, $\hat f \in W(C(\mathbb R),\ell^1(\mathbb Z))$, see Section \ref{sec:samplingpolynomials} for the definition,
%such that the scaled maximum approximation error $M_{P,L}(f)$ is small? We measure the decay behavior of $f$ and $\hat f$
%by the corresponding decay rates \eqref{eq:decayrate}.
%\item How can ${\hat f}(v)$ for $|v| \le P$, be rapidly computed by NFFT with small approximation error $M_{P,L}(f)$?
%\end{enumerate}
\begin{enumerate}
	\item Under which decay assumptions on $f$ and $\hat f$ can the Fourier transform
	$\hat f$ be uniformly approximated by the trigonometric sampling polynomial
	$s_{P,L}$ with explicitly controllable error?
	
	\item How can the NFFT be employed to compute values of the continuous Fourier
	transform $\hat f(v)$ efficiently and with rigorous $\ell_\infty$ error bounds?
\end{enumerate}

We introduce the so-called decay rate $\delta(f,L)$ of $f \in W(C(\mathbb R), \ell^1(\mathbb Z))$ with respect to step
size $L \in 2\mathbb N$, which is defined by \eqref{eq:decayrate}. If $f \in C(\mathbb R)$ has polynomial or exponential decay
(see Example \ref{Ex:poly/expodecay}), then the corresponding decay rate $\delta(f,L)$ can be easily estimated (see Lemma \ref{Lemma:polydecay}
and Lemma \ref{Lemma:expodecay}). Applying the Poisson summation formula \eqref{eq:Poisson2}, we show in Theorem \ref{Thm:samplingpolynomials}
that an upper bound of $M_{P,L}(f)$ is a linear combination of the decay rates $\delta(\hat f,P)$ and $\delta(f,L)$, if
$f$, $\hat f \in W(C(\mathbb R),\ell^1(\mathbb Z))$ and $P$, $L\in 2 \mathbb N$ are given.

Theorem \ref{Thm:samplingpolynomials} shows that the Fourier transform
$\hat f$ can be approximated uniformly by the trigonometric sampling polynomial
$s_{P,L}$ with an error controlled by the decay rates of $f$ and $\hat f$.
Since $s_{P,L}$ is a trigonometric polynomial, its values can be computed
efficiently by FFT or NFFT techniques. Consequently, the theorem provides a
rigorous foundation for the computation of continuous Fourier transforms by
means of the NFFT.

\medskip

\new{The main contributions of this paper can be summarized as follows:
\begin{enumerate}[label=(C\arabic*)]
\item We introduce the decay rate $\delta(f,L)$ in \eqref{eq:decayrate} and prove
in Theorem~\ref{Thm:samplingpolynomials} the \emph{uniform} error bound
$M_{P,L}(f) \le L^{-1/2}\,\delta(\hat f,P) + L^{1/2}\,\delta(f,L)$, valid on the
whole interval $[-P/2,\,P/2]$ and not only on the grid $\{k/L\}$. This is the
property which is needed if $s_{P,L}$ is evaluated at \emph{arbitrary} nodes by
the NFFT.
\item All constants in the resulting estimates
\eqref{eq:maxerrorpolydecay}, \eqref{eq:maxerrorexpodecay} and
\eqref{eq:maxerrormixeddecay} are explicit and are computed directly from the
decay parameters of $f$ and $\hat f$. Moreover, the exponential case is treated
for all $\alpha,\,\beta > 0$, whereas \cite{EhGrKl24} is restricted to
$0 < \alpha,\,\beta \le 1$.
\item We prove in Theorem~\ref{Thm:lowerbound} a matching \emph{lower} bound,
which shows that the rate $L^{1/2-a}$ in the truncation parameter cannot be
improved, and we determine in Remark~\ref{Rem:notequivalent} the exact range
$\bigl[1,\,\sqrt{LP}\,\bigr]$ of the quotient
$E_{P,L}^{\ast}(f)/M_{P,L}^{\ast}(f)$. 
\item We state the resulting method explicitly as
Algorithm~\ref{Alg:FTviaNFFT}, including its arithmetic cost and its a~priori
error bound, and we compare direct evaluation, FFT and NFFT numerically in
Section~\ref{sec:cost}.
\end{enumerate}
\medskip
}

	This paper is organized as follows:
\new{Section~\ref{sec:literature} reviews the relevant literature and states
precisely which of our results are new.}
In Section~\ref{sec:samplingpolynomials}, we recall
	the Wiener amalgam space $W(C(\mathbb{R}),\ell^1(\mathbb{Z}))$, introduce the decay rate
	$\delta(f,L)$, and derive uniform error estimates for the scaled maximum approximation
	error $M_{P,L}(f)$ under polynomial, exponential, and mixed decay conditions on $f$ and
	$\hat{f}$\new{, together with a matching lower bound and the resulting
	Algorithm~\ref{Alg:FTviaNFFT}}.
Section~\ref{sec:numerics} demonstrates how the theoretical error bounds
translate into practical parameter choices for FFT- and NFFT-based
approximations of the Fourier transform. In particular, the numerical
experiments confirm that the NFFT yields highly accurate evaluations of
$\hat f(v)$ throughout the frequency interval while preserving the expected
computational efficiency.
\new{Section~\ref{sec:cost} is devoted to the computational cost and compares
direct evaluation, FFT and NFFT. Section~\ref{sec:conclusion} concludes the
paper.}

\new{
\section{Literature review\label{sec:literature}}

In this section we relate the present paper to the state of the art. We group
the discussion into three strands -- the approximation of the Fourier transform
by discrete sums, the function-space framework, and fast algorithms -- and
close with a precise account of what is new here.

\subsection{Approximation of the Fourier transform by discrete sums}

Replacing the Fourier integral \eqref{eq:hatf} by a finite sum of the form
\eqref{eq:gell} is the oldest and most widely used numerical device in Fourier
analysis; classical accounts are given in \cite[pp.~17--23]{BH95} and
\cite[pp.~507--508]{PlPoStTa23}. The two error sources -- \emph{aliasing},
caused by the discretization of the integral with step size $1/P$, and
\emph{truncation}, caused by restricting the sum to the window
$[-L/2,\,L/2]$ -- have long been understood qualitatively. Quantitative,
non-asymptotic statements have been obtained only recently.

The paper of M.~Ehler, K.~Gr\"ochenig, and A.~Klotz \cite{EhGrKl24} is the direct starting point of the
present work. The authors measure the accuracy of \eqref{eq:gell} by the scaled
Euclidean error $E_{P,L}^{\ast}(f)$ and prove, for $f$ and $\hat f$ in Wiener
amalgam spaces, bounds of the form $C\,(L^{-\alpha} + P^{-\beta})$ with
$\alpha < a - \frac12$ and $\beta < b - \frac12$ in the case of polynomial
decay, and of the form
$C\,({\mathrm e}^{-r'(L/2)^{\alpha}} + {\mathrm e}^{-s'(P/2)^{\beta}})$ with
$r' < r$, $s' < s$ and $0 < \alpha,\,\beta \le 1$ in the case of exponential
decay. Their results are asymptotically sharp in the exponents, and they give
an asymptotically optimal recipe for balancing the number of samples, the
sampling step and the grid size. The loss  in the exponents is intrinsic to their argument.

Closely related questions arise in the numerical realization of Shannon's
sampling theorem, where a bandlimited function is reconstructed from finitely
many samples. Sharp bounds for the norm of the Shannon sampling operator and
regularized sampling formulas with sinh-type or continuous Kaiser--Bessel
windows are studied in \cite{KiPoTa22, KiPoTa23}; see also the earlier work \cite{Q03,StTa06}.
These papers are concerned with the reconstruction of $f$ from its samples,
i.e.\ with the situation of Remark~\ref{Rem:maxPoisson1}, whereas we are
interested in the forward problem, the evaluation of $\hat f$.

\subsection{The function-space framework}

The natural setting for the Poisson summation formula
\eqref{eq:Poisson1}--\eqref{eq:Poisson2} is the Wiener amalgam space
$W(C(\mathbb{R}),\ell^1(\mathbb{Z}))$ introduced by N.~Wiener
\cite[p.~73]{Wie88} and developed systematically by H.G.~Feichtinger
\cite{Fe92a}; see also \cite[pp.~103--105]{G01}. The example of
Y.~Katznelson \cite{K67} shows that the space
$C(\mathbb{R}) \cap L^1(\mathbb{R})$ is genuinely too large, and
K.~Gr\"ochenig \cite{G96} clarified the role of the underlying uncertainty
principle. We adopt this framework unchanged from \cite{EhGrKl24}, since it is
exactly the setting in which the two Poisson summation formulae hold pointwise
and with absolute and uniform convergence, which is what a uniform error bound
requires. 

\subsection{Fast algorithms}

The NFFT goes back to A.~Dutt and V.~Rokhlin \cite{duro93}, G.~Beylkin
\cite{bey95} and G.~Steidl \cite{st97}; a survey is \cite{GrLe04} and a textbook
treatment is \cite[Chapter~7]{PlPoStTa23}. Widely used software libraries are
NFFT3 \cite{nfft3,KeKuPo09} and FINUFFT \cite{FINUFFT,FINUFFTpaper}; the
experiments in Section~\ref{sec:cost} use the Python interface
\texttt{pyNFFT3} \cite{pyNFFT3} of NFFT3.

The accuracy of the NFFT itself -- i.e.\ the error committed when the
trigonometric polynomial \eqref{eq:trig_poly} with \emph{given} coefficients is
evaluated approximately -- is by now well understood. Uniform error estimates
for compactly supported window functions are given in \cite{PoTa21a}, continuous
window functions are analysed in \cite{PoTa21b}, and the case of nonequispaced
data in both domains, together with fast sinc transforms, is treated in
\cite{KiPoTa23a}. The FINUFFT kernel is analysed in \cite{Ba21}. We stress that
this line of work is complementary to ours: it bounds the difference between
$s_{P,L}$ and its NFFT approximation, whereas we bound the difference between
$\hat f$ and $s_{P,L}$. Only the combination of both, as in
Algorithm~\ref{Alg:FTviaNFFT} and Remark~\ref{Rem:totalerror}, yields a
rigorous bound for the quantity a user actually computes. 

\subsection{What is new in this paper\label{sec:whatisnew}}

Since the relation to \cite{EhGrKl24} is close, we state explicitly which of
our results are new, which cannot be obtained from \cite{EhGrKl24}, and which
assumptions have been weakened. A synopsis is given in
Table~\ref{tab:comparison}.

\begin{enumerate}[label=(N\arabic*), leftmargin=*]
\item \emph{A different, strictly finer error measure.} We estimate the
maximum error $M_{P,L}(f)$ over the \emph{whole interval} $[-P/2,\,P/2]$, while
\cite{EhGrKl24} estimates the Euclidean error $E_{P,L}^{\ast}(f)$ over the
\emph{finite grid} $\{k/L:\,k \in [LP]\}$. A maximum error does occur in
\cite[Section~2.7, formula~(12)]{EhGrKl24}, but only through the trivial
estimate $M_{P,L}^{\ast}(f) \le E_{P,L}^{\ast}(f)$, i.e.\ the discrete maximum
error is bounded there by the Euclidean error and therefore inherits its rate.
This is exactly the step which is lossy: the two quantities satisfy
$$
M_{P,L}^{\ast}(f) \;\le\; E_{P,L}^{\ast}(f)
\;\le\; \sqrt{L P}\; M_{P,L}^{\ast}(f)\,,
$$
and both inequalities are attained up to constants, see
Remark~\ref{Rem:notequivalent} and Figure~\ref{fig:EstarM}. The two error
measures are therefore \emph{not} equivalent: their quotient is unbounded as
$LP \to \infty$. Consequently, Theorem~\ref{Thm:polydecay} cannot be deduced
from \cite[Theorem~2.1]{EhGrKl24} via \cite[(12)]{EhGrKl24}; that route yields
only the exponent $P^{-b+1/2+\varepsilon}$ instead of $P^{-b}$, and the gap is
clearly visible in Figure~\ref{fig:EstarM}. Conversely, an estimate on a finite
grid carries no information about the values of $\hat f - s_{P,L}$ between the
grid points, which is precisely the information needed when $s_{P,L}$ is
evaluated by an NFFT at arbitrary nodes.

\item \emph{Explicit constants.} All constants in
\eqref{eq:maxerrorpolydecay}, \eqref{eq:maxerrorexpodecay} and
\eqref{eq:maxerrormixeddecay} are given in closed form in terms of the decay
parameters $a,\,b,\,c,\,d,\,r,\,s,\,\alpha,\,\beta$. This makes the estimates
usable as an a~priori stopping criterion, i.e.\ for the actual choice of $L$
and $P$ (see Algorithm~\ref{Alg:FTviaNFFT}).

\item \emph{No loss in the exponents, and a weaker assumption in the
exponential case.} In the polynomial case we obtain the exponent $b$ in $P$ and
the exponent $a - \frac12$ in $L$, whereas \cite{EhGrKl24} must give away an
arbitrarily small $\varepsilon > 0$ in both exponents. In the exponential case,
\cite[Theorems~2.3 and 2.4]{EhGrKl24} assume $0 < \alpha,\,\beta \le 1$ and
lose an arbitrarily small amount in the decay parameters $r$ and $s$;
Theorem~\ref{Thm:expodecay} and Theorem~\ref{Thm:mixeddecay} hold for
\emph{all} $\alpha,\,\beta > 0$ and retain the original $r$ and $s$. In
particular, the Gaussian ($\alpha = \beta = 2$) of Example~\ref{Ex:Gaussian} is
not covered by \cite{EhGrKl24}.

\item \emph{Sharpness.} Theorem~\ref{Thm:lowerbound} provides a lower bound
which matches the upper bound of Theorem~\ref{Thm:polydecay} in the truncation
parameter $L$ up to a constant factor. In \cite{EhGrKl24} the sharpness of the
estimates is conjectured and confirmed by numerical experiments, but no lower
bound is proved there.

\item \emph{Algorithm and computational assessment.}
Algorithm~\ref{Alg:FTviaNFFT} combines the estimate of
Theorem~\ref{Thm:samplingpolynomials} with
the NFFT and states the total error and the arithmetic cost;
Section~\ref{sec:cost} compares direct evaluation, FFT and NFFT.
\end{enumerate}

The general strategy of splitting the
Poisson summation formula into a principal part and two remainders, and the use
of Wiener amalgam spaces, are taken from \cite{EhGrKl24}, as is the scaling
$L^{-1/2}$ in the definition \eqref{eq:uniform} of the error, which we have
kept in order to make the two sets of results directly comparable.

\begin{table}[ht]
\centering
\captionsetup{font=small}
\caption{Comparison of \cite{EhGrKl24} with the present paper. Here
$a,\,b > 1$ denote the polynomial and $r,\,\alpha,\,s,\,\beta > 0$ the
exponential decay parameters of $f$ and $\hat f$; $\varepsilon > 0$ is
arbitrarily small.}
\label{tab:comparison}
\small
\begin{tabular}{@{}p{0.27\textwidth}p{0.335\textwidth}p{0.335\textwidth}@{}}
\hline
& \cite{EhGrKl24} & this paper \\
\hline
error measure
& $E_{P,L}^{\ast}(f)$, Euclidean, on the grid $\{k/L\}$; $M_{P,L}^{\ast}$ only
  via $M_{P,L}^{\ast}\le E_{P,L}^{\ast}$
& $M_{P,L}(f)$, maximum, on all of $[-P/2,\,P/2]$ \\[0.6ex]
constants
& not specified
& explicit in $a,b,c,d,r,s,\alpha,\beta$ \\[0.6ex]
polynomial decay
& $\mathcal O\bigl(L^{-a+1/2+\varepsilon}$ \newline
  \hspace*{1.2em}$+\ P^{-b+1/2+\varepsilon}\bigr)$
& $\mathcal O\bigl(L^{-a+1/2}$ \newline
  \hspace*{1.2em}$+\ L^{-1/2}P^{-b}\bigr)$, Thm.~\ref{Thm:polydecay} \\[0.6ex]
exponential decay
& only $0 < \alpha,\,\beta \le 1$, rates $r' < r$, $s' < s$
& all $\alpha,\,\beta > 0$, rates $r$, $s$,
  Thm.~\ref{Thm:expodecay} \\[0.6ex]
lower bounds
& conjectured, numerical evidence
& proved, Thm.~\ref{Thm:lowerbound} \\[0.6ex]
evaluation at arbitrary $v$
& not covered
& Alg.~\ref{Alg:FTviaNFFT}, Rem.~\ref{Rem:totalerror} \\[0.6ex]
cost comparison
& --
& Section~\ref{sec:cost} \\
\hline
\end{tabular}
\end{table}
}

\section{Uniform approximation by trigonometric sampling polynomials
	\label{sec:samplingpolynomials}}

The \emph{Wiener amalgam space} $W(C(\mathbb{R}),\ell^1(\mathbb{Z}))$ is the
Banach space consisting of all functions $f:\,\mathbb{R} \to \mathbb{C}$ which
are \emph{locally} in $C(\mathbb{R})$ and have \emph{globally} an
$\ell^1(\mathbb{Z})$ behavior at infinity in the sense that the uniform norms
of $f$ over the intervals $[k,\,k+1]$, $k \in \mathbb{Z}$, form a sequence in
$\ell^1(\mathbb{Z})$. The norm of $W(C(\mathbb{R}),\ell^1(\mathbb{Z}))$ is
given by
$$
\|f\|_{W(C(\mathbb{R}),\ell^1(\mathbb{Z}))}
:= \sum_{k \in \mathbb{Z}} \max_{x \in [0,\,1]} |f(x + k)|\,.
$$
The Wiener amalgam space was introduced by N.~Wiener \cite[p.~73]{Wie88} and was used as convenient function space in Fourier analysis, cf. \cite{Fe92a}, \cite[pp.~103--105]{G01}, and \cite{EhGrKl24}.
The condition $f \in W(C(\mathbb{R}),\ell^1(\mathbb{Z}))$ is strong enough to
exclude many pathological functions which play a role in Fourier analysis but
are of little practical interest. The Wiener amalgam space
$W(C(\mathbb{R}),\ell^1(\mathbb{Z}))$ contains all continuous, compactly
supported functions and is therefore a dense subspace of the Lebesgue space
$L^1(\mathbb{R})$. This Wiener amalgam space has the following property:

\begin{lemma}$\mathrm{(see\;\cite[Lemma~6.1.2]{G01})}$
	\label{Lemma:PropWiener}
	If $f \in W(C(\mathbb{R}),\ell^1(\mathbb{Z}))$ and $R > 0$ are given, then
	$$
	\sup_{x \in \mathbb{R}} \sum_{k \in \mathbb{Z}} |f(x + Rk)|
	\leq \left(\frac{1}{R} + 1\right)\|f\|_{W(C(\mathbb{R}),\ell^1(\mathbb{Z}))}\,.
	$$
\end{lemma}

From Lemma~\ref{Lemma:PropWiener} with $R = \tfrac{1}{L}$, $L \in 2 \mathbb N$, it follows immediately that
\begin{equation}
	\label{eq:sum|f(Rk)|}
	\sum_{k \in \mathbb{Z}} \big|f\big(\tfrac{k}{L}\big)\big|
	\leq \left(\frac{1}{R} + 1\right)\|f\|_{W(C(\mathbb{R}),\ell^1(\mathbb{Z}))} < \infty
\end{equation}
such that $\bigl( f\big(\tfrac{k}{L}\big)\bigr)_{k \in \mathbb{Z}} \in \ell^1(\mathbb{Z})$ for each $L \in 2 \mathbb N$.
\medskip

For each $f \in W(C(\mathbb R), \ell^1(\mathbb Z))$ we introduce the \emph{decay rate of} $f$ \emph{with respect to the step size} $L \in 2\mathbb N$
by
\begin{equation}
\label{eq:decayrate}
\delta(f,L) := \max_{|x|\leq L/2} \sum_{k \in \mathbb Z \setminus \{0\}} |f(x + k L)|\,.
\end{equation}
From Lemma \ref{Lemma:PropWiener} it follows immediately that $\delta(f,L) < \infty$ for arbitrary $L \in 2 \mathbb N$.
\medskip

\begin{example}
\label{Ex:poly/expodecay}
	A function $f$ belongs to the Wiener amalgam space
	$W(C(\mathbb{R}),\ell^1(\mathbb{Z}))$ if and only if $f \in C(\mathbb{R})$ can
	be majorized by the step function $s$ with jumps at integers, given by
	$s(x) = m_k$ for $x \in [k,\,k+1)$, $k \in \mathbb{Z}$, with
	$m_k := \max_{x \in [k,\,k+1]}|f(x)|$ and $\sum_{k \in \mathbb{Z}} m_k < \infty$. Hence it holds
$$
W(C(\mathbb{R}),\ell^1(\mathbb{Z})) \subset C(\mathbb R) \cap L^1(\mathbb R)\,.
$$
	In particular, each compactly supported function $f \in C(\mathbb{R})$ belongs
	to $W(C(\mathbb{R}),\ell^1(\mathbb{Z}))$. If $f \in C(\mathbb{R})$ has
	\emph{polynomial decay}, i.e.,
	$$
	\sup_{x \in \mathbb{R}} |f(x)|\,(1 + |x|)^a \leq c < \infty\,, \quad a > 1\,,
	$$
	or \emph{exponential decay}, i.e.,
	$$
	\sup_{x \in \mathbb{R}} |f(x)|\,{\mathrm{e}}^{r\,|x|^{\alpha}} \leq d < \infty\,,
	\quad r,\,\alpha > 0\,,
	$$
	then $f$ belongs to $W(C(\mathbb{R}),\ell^1(\mathbb{Z}))$. For corresponding proofs see Lemma \ref{Lemma:polydecay} and Lemma \ref{Lemma:expodecay}.
\end{example}

\subsection{Poisson summation formula}

The key tool for estimating the uniform approximation error is the Poisson
summation formula. It allows us to express the error between $\hat{f}$ and the
trigonometric sampling polynomial $s_{P,L}$ in terms of the tails of $f$ and
$\hat{f}$, which can then be controlled by decay conditions on these functions.

\begin{lemma}$\mathrm{(see\;\cite{G96})}$
	\label{Lemma:Poissonsum}
	Let $P \in 2 \mathbb N$ be fixed. For each $f \in W(C(\mathbb{R}),\ell^1(\mathbb{Z}))$ with
	$\hat{f} \in W(C(\mathbb{R}),\ell^1(\mathbb{Z}))$, the \emph{Poisson
		summation formulae}
	\begin{eqnarray}\label{eq:Poisson1}
		\sum_{k \in \mathbb{Z}} f(x + kP) &=& \frac{1}{P}\sum_{n \in \mathbb{Z}}
		\hat{f}\!\left(\frac{n}{P}\right){\mathrm{e}}^{2\pi{\mathrm{i}}\,nx/P}\,,
		\\
		\sum_{k \in \mathbb{Z}} \hat{f}(v + kP) &=& \frac{1}{P}\sum_{n \in \mathbb{Z}}
		f\!\left(\frac{n}{P}\right){\mathrm{e}}^{-2\pi{\mathrm{i}}\,nv/P}
		\label{eq:Poisson2}
	\end{eqnarray}
	hold for all $x,\, v \in \mathbb{R}$. All
	series converge absolutely and uniformly on $\mathbb{R}$.
\end{lemma}

\emph{Proof}. From
Lemma~\ref{Lemma:PropWiener} it follows that
$$
f_P(x) := \sum_{k \in \mathbb{Z}} f(x + kP)\,, \quad x \in \mathbb{R}\,,
$$
is a continuous, $P$-periodic function. By the assumption
$\hat{f} \in W(C(\mathbb{R}),\ell^1(\mathbb{Z}))$, it follows from
\eqref{eq:sum|f(Rk)|} that
$$
\sum_{k \in \mathbb{Z}} \left|\hat{f}\!\left(\frac{n}{P}\right)\right| < \infty\,,
$$
so that the $P$-periodic Fourier series
$\frac{1}{P}\sum_{n \in \mathbb{Z}} \hat{f}\!\left(\frac{n}{P}\right)
{\mathrm{e}}^{2\pi{\mathrm{i}}\,nx/P}$ converges absolutely and uniformly on
$\mathbb{R}$. Both functions coincide, since $f_P$ has the $P$-periodic Fourier
coefficients
$$
\frac{1}{P}\int_0^P f_P(x)\,{\mathrm{e}}^{-2\pi{\mathrm{i}}\,nx/P}\,{\mathrm{d}}x
= \frac{1}{P}\int_{\mathbb{R}} f(x)\,{\mathrm{e}}^{-2\pi{\mathrm{i}}\,nx/P}\,
{\mathrm{d}}x = \frac{1}{P}\,\hat{f}\!\left(\frac{n}{P}\right),
\quad n \in \mathbb{Z}\,.
$$
This completes the proof of \eqref{eq:Poisson1}. Replacing $f$ by $\hat{f}$
yields \eqref{eq:Poisson2}. \qedsymbol
\medskip

\begin{remark}
	By an example, Y.~Katznelson \cite{K67} has shown that the Poisson summation
	formula \eqref{eq:Poisson1} does not hold for all
	$f \in C(\mathbb{R}) \cap L^1(\mathbb{R})$ with
	$\hat{f} \in C(\mathbb{R}) \cap L^1(\mathbb{R})$, illustrating the essential
	difference between the spaces $W(C(\mathbb{R}),\ell^1(\mathbb{Z}))$ and
	$C(\mathbb{R}) \cap L^1(\mathbb{R})$.
\end{remark}
\medskip

\subsection{The trigonometric sampling polynomial and its fast evaluation}

The trigonometric sampling polynomial provides a natural approximation of the
Fourier transform $\hat{f}$ on a compact frequency interval, and can be
evaluated efficiently at arbitrary points using the NFFT. Now we make this
precise.

For $f \in W(C(\mathbb{R}), \ell^1(\mathbb{Z}))$ with
$\hat{f} \in W(C(\mathbb{R}), \ell^1(\mathbb{Z}))$ and sufficiently large
$L,\, P \in 2\mathbb{N}$, we define the $P$-\emph{periodic trigonometric
	sampling polynomial of} $f$ \emph{with sampling support} $[-L/2, L/2]$
\emph{and degree} $LP/2$ by $s_{P,L}$, see \eqref{eq:samplingpoly}.
%\begin{equation}
%	\label{eq:samplingpoly}
%	s_{P,L}(v) := \frac{1}{P}\sum_{n \in [LP]} f\!\left(\frac{n}{P}\right)
%	{\mathrm{e}}^{-2\pi{\mathrm{i}}\,nv/P}\,, \quad v \in \mathbb{R}\,.
%\end{equation}
Under suitable decay conditions on $f$ and $\hat{f}$, the polynomial
$s_{P,L}$ provides a good uniform approximation of $\hat{f}$ on the compact
interval $[-P/2,\, P/2]$, see Theorem~\ref{Thm:samplingpolynomials} and the
following results.

The values $s_{P,L}(k/L)$, $k \in [LP]$, on the uniform frequency grid can be
computed via the FFT in $\mathcal{O}(LP\log LP)$ operations, provided $L$ and
$P$ are products of powers of small primes (see \cite{FJ05}). More generally,
the trigonometric sampling polynomial $s_{P,L}$ can be evaluated at $K$ arbitrary points
$v_1,\ldots,v_K \in [-P/2,\, P/2]$ by means of the nonequispaced fast Fourier
transform (NFFT) in only $\mathcal{O}(LP\log LP + K)$ operations
(see \cite{PlPoStTa23}). Hence, for each $v$ with $|v| \leq P/2$, the value
$s_{P,L}(v)$ is a computationally efficient approximation of $\hat{f}(v)$.

\subsection{The uniform approximation error}

Let $L$, $P \in 2 \mathbb N$ be given.
By convenient decay conditions on $f$ and $\hat{f}$, we can show that
$s_{P,L}$ is a good uniform approximation of $\hat{f}$ on the compact interval
$\bigl[-\frac{P}{2},\,\frac{P}{2}\bigr]$. To this end, we use the
scaled uniform approximation error \eqref{eq:uniform}.
If we form the maximum only on the grid $\big\{\frac{k}{L}:\,k\in [L P]\big\}$, then we obtain the discretized version of the scaled uniform approximation error \eqref{eq:discreteuniform}.
We derive explicit estimates for $M_{P,L}(f)$ and $M_{P,L}^{\ast}(f)$ under various decay conditions
on $f$ and $\hat{f}$. The key decomposition is provided by
Theorem~\ref{Thm:samplingpolynomials} below, which splits the error into two
contributions: one governed by the decay of $\hat{f}$ and one by the decay of
$f$. Explicit bounds for each contribution are then given in
Lemma~\ref{Lemma:polydecay} and Lemma~\ref{Lemma:expodecay}, covering the
cases of polynomial and exponential decay respectively.

\begin{theorem}
	\label{Thm:samplingpolynomials}
	Let $L$, $P \in 2\mathbb{N}$ be fixed. Let
	$f \in W(C(\mathbb{R}), \ell^1(\mathbb{Z}))$ with
	$\hat{f} \in W(C(\mathbb{R}), \ell^1(\mathbb{Z}))$ be given.\\
Then $\hat{f}$
	can be uniformly approximated on $[-P/2, P/2]$ by the trigonometric sampling
	polynomial $s_{P,L}$, and it holds the estimate
	\begin{equation}
		\label{eq:fsuni}
		M_{P,L}^{\ast}(f) \leq M_{P,L}(f) \leq \frac{1}{\sqrt L}\, \delta(\hat f, P) + \sqrt L\,\delta(f,L)\,,
	\end{equation}
	where $\delta(\hat f, P)$ and $\delta(f,L)$ denote the corresponding decay rates \eqref{eq:decayrate}.
\end{theorem}

\emph{Proof}. Splitting both series in the second Poisson summation formula \eqref{eq:Poisson2} into the principal
terms and the remainders, we obtain for all $v \in \mathbb{R}$
$$
\hat{f}(v) + \sigma_1(v)= s_{P,L}(v) + \sigma_2(v)\,,
$$
with
\begin{eqnarray*}
\sigma_1(v) &:=& \sum_{k \in \mathbb{Z} \setminus \{0\}} \hat{f}(v + k P)\,, \\
\sigma_2(v) &:=& \frac{1}{P}\sum_{n \in \mathbb{Z} \setminus [LP]} f\!\left(\frac{n}{P}\right){\mathrm{e}}^{-2\pi{\mathrm{i}}\,nv/P}\,.
\end{eqnarray*}
Hence this provides that
$$
M_{P,L}^{\ast}(f) \leq M_{P,L}(f) \leq \frac{1}{\sqrt L} \max_{|v|\leq P/2} |\sigma_1(v)| + \frac{1}{\sqrt L} \max_{|v|\leq P/2} |\sigma_2(v)|\,.
$$
Applying the triangle inequality and \eqref{eq:decayrate}, we obtain
$$
\max_{|v|\leq P/2} |\sigma_1(v)| \leq \max_{|v| \leq P/2} \sum_{k \in \mathbb{Z} \setminus \{0\}} |\hat{f}(v + k P)| = \delta(\hat f, P)\,.
$$
In the expression $\sigma_2$, the
summation indices $n \in \mathbb{Z} \setminus [LP]$ can be written in the form
$n = m + kLP$ with $m \in [LP]$ and $k \in \mathbb{Z} \setminus \{0\}$, such that
$$
\sigma_2(v) = \frac{1}{P}\sum_{m \in [LP]}\,
\sum_{k \in \mathbb{Z} \setminus \{0\}}
f\!\left(\frac{m}{P} + kL\right){\mathrm{e}}^{-2\pi{\mathrm{i}}\,mv/P}\,{\mathrm{e}}^{-2\pi{\mathrm{i}}\,k L v} \,.
$$
Using triangle inequality and \eqref{eq:decayrate}, this yields
$$
\max_{|v| \leq P/2} |\sigma_2(v)| \leq \frac{L P}{P}\, \max_{m \in [L P]} \sum_{k \in \mathbb{Z} \setminus \{0\}}\big|f\!\left(\frac{m}{P} + kL\right)\big| \leq L\,\delta(f,L)\,.
$$
This completes the proof. \qedsymbol
\medskip

\begin{remark}
	\label{Rem:maxPoisson1}
	For given $f,\, \hat{f} \in W(C(\mathbb{R}),\ell^1(\mathbb{Z}))$, we can
	analogously estimate the maximum approximation error
	\begin{equation}
		\label{eq:maxerror1}
		\frac{1}{\sqrt L}\,\max_{|x| \leq P/2} \left|f(x) - \frac{1}{P}\sum_{n \in [LP]}
		\hat{f}\!\left(\frac{n}{P}\right){\mathrm{e}}^{2\pi{\mathrm{i}}\,n x/P}\right|
	\end{equation}
	by applying the first Poisson summation formula \eqref{eq:Poisson1}. The
	error \eqref{eq:maxerror1} has the upper bound
	$$
	\frac{1}{\sqrt L}\, \delta(f,P) + \sqrt L \, \delta(\hat f, L)\,,
	$$
	which can be further estimated under concrete decay conditions on $f$ and
	$\hat{f}$ by Lemma~\ref{Lemma:polydecay} and Lemma~\ref{Lemma:expodecay}.
For the discretized version of \eqref{eq:maxerror1} we obtain
$$
\frac{1}{\sqrt L}\,\max_{\ell \in [L P]} \left|f\left(\frac{\ell}{L}\right) - \frac{1}{P}\sum_{n \in [LP]}
		\hat{f}\!\left(\frac{n}{P}\right){\mathrm{e}}^{2\pi{\mathrm{i}}\,n \ell/(L P)}\right|  \leq \frac{1}{\sqrt L}\, \delta(f,P) + \sqrt L \, \delta(\hat f, L)\,.
$$
\end{remark}
\medskip

\begin{remark}
	\label{Rem:inverse}
	The reconstruction in Remark~\ref{Rem:maxPoisson1} presupposes the exact
	Fourier samples $\hat f(n/P)$, $n \in [LP]$. In practice, however, only the
	computed trigonometric sampling polynomial $s_{P,L}$, see
	\eqref{eq:samplingpoly}, is available. We therefore sample $s_{P,L}$ on the
	finer equispaced grid $\bigl\{j/M:\, j \in [MP]\bigr\}$ of the frequency
	interval $[-P/2,\, P/2)$, where $M \in 2\mathbb{N}$ with $M \geq 2L$ denotes a
	frequency oversampling factor, and form the fully computable approximation
	\begin{equation}
		\label{eq:invapprox}
		f_{M,P,L}(x) := \frac{1}{M} \sum_{j \in [MP]}
		s_{P,L}\!\left(\frac{j}{M}\right) {\mathrm{e}}^{2\pi{\mathrm{i}}\,jx/M}\,,
		\qquad x \in [-L/2,\, L/2)\,.
	\end{equation}
	This is the $M$-point rectangle-rule discretization of the truncated inverse
	Fourier integral $\int_{-P/2}^{P/2} s_{P,L}(v)\,{\mathrm{e}}^{2\pi{\mathrm{i}}\,vx}\,
	{\mathrm{d}}v$; on the grid $x = \ell/P$, $\ell \in [LP]$, it can be evaluated by a
	single inverse \FFT{} of length $MP$.

	Introducing the truncated cardinal series
	\begin{equation}
		\label{eq:gPL}
		g_{P,L}(x) := \int_{-P/2}^{P/2} s_{P,L}(v)\,{\mathrm{e}}^{2\pi{\mathrm{i}}\,vx}\,
		{\mathrm{d}}v = \sum_{n \in [LP]} f\!\left(\frac{n}{P}\right)
		{\mathrm{sinc}}\bigl(Px - n\bigr)\,, \qquad
		{\mathrm{sinc}}(t) := \frac{\sin(\pi t)}{\pi t}\,,
	\end{equation}
	the approximation error splits into three contributions,
	\begin{align*}
	f(x) - f_{M,P,L}(x) &= \underbrace{\int_{|v| > P/2} \hat f(v)\,
		{\mathrm{e}}^{2\pi{\mathrm{i}}\,vx}\,{\mathrm{d}}v}_{\text{band truncation}}
	+ \underbrace{\int_{-P/2}^{P/2} \bigl(\hat f - s_{P,L}\bigr)(v)\,
		{\mathrm{e}}^{2\pi{\mathrm{i}}\,vx}\,{\mathrm{d}}v}_{\text{sampling polynomial}}\\
	&+ \underbrace{\bigl(g_{P,L}(x) - f_{M,P,L}(x)\bigr)}_{\text{frequency discretization}}\,.
		\end{align*}
	The first term is bounded by the $L^1$-tail
	$\eta(\hat f, P) := \int_{|v| > P/2} |\hat f(v)|\,{\mathrm{d}}v$, the second by
	$P\,\max_{|v| \leq P/2} |\hat f(v) - s_{P,L}(v)| = P\sqrt L\,M_{P,L}(f)$, see
	\eqref{eq:uniform}. For the third term, a direct evaluation of the geometric sum
	in \eqref{eq:invapprox} gives the closed form
	$$f_{M,P,L}(x) = \frac{1}{MP}\sum_{n \in [LP]} f(n/P)\,
	{\mathrm{e}}^{{\mathrm{i}}\pi\alpha_n/M}\,
	\frac{\sin(\pi P\alpha_n)}{\sin(\pi\alpha_n/M)}$$ with $\alpha_n := x - n/P$, hence
	$$
	g_{P,L}(x) - f_{M,P,L}(x) = \frac{1}{P} \sum_{n \in [LP]} f\!\left(\frac{n}{P}\right)
	\sin(\pi P\alpha_n)\left[\frac{1}{\pi\alpha_n}
	- \frac{{\mathrm{e}}^{{\mathrm{i}}\pi\alpha_n/M}}{M\,\sin(\pi\alpha_n/M)}\right].
	$$
	Since $|\sin(\pi P\alpha_n)| \leq 1$ and, by the elementary inequality
	$1 - \beta\cot\beta \leq \beta^2$ for $|\beta| \leq \pi/2$, the bracket is
	bounded in modulus by $M^{-1}\sqrt{1 + \pi^2/4}$ whenever $|\alpha_n| \leq M/2$
	(which holds for $|x| \leq L/2$ as $M \geq 2L$), we obtain with
	Lemma~\ref{Lemma:PropWiener} (for $R = 1/P$ and $|x| \leq L/2$)
	$$
	\bigl|g_{P,L}(x) - f_{M,P,L}(x)\bigr| \leq
	\frac{\sqrt{1 + \pi^2/4}}{M P} \sum_{n \in [LP]} \left|f\!\left(\frac{n}{P}\right)\right|
	\leq \frac{(P+1)\sqrt{1 + \pi^2/4}}{M P}\,
	\|f\|_{W(C(\mathbb{R}),\ell^1(\mathbb{Z}))}\,.
	$$
	Combining the three bounds with Theorem~\ref{Thm:samplingpolynomials} and
	$(P+1)/P \leq 2$ yields the uniform error estimate
	\begin{equation}
		\label{eq:inverr}
		\max_{|x| \leq L/2} \bigl|f(x) - f_{M,P,L}(x)\bigr| \leq
		\eta(\hat f, P) + P\,\bigl(\delta(\hat f, P) + L\,\delta(f, L)\bigr)
		+ \frac{2\sqrt{1 + \pi^2/4}}{M}\,\|f\|_{W(C(\mathbb{R}),\ell^1(\mathbb{Z}))}\,.
	\end{equation}
	The three terms reflect, respectively, the truncation of the inverse transform
	to the band $[-P/2,\, P/2]$, the approximation of $\hat f$ by $s_{P,L}$ on that
	band, and the rectangle-rule discretization in the frequency variable.

	If $\hat f$ has polynomial decay $\sup_{v} |\hat f(v)|\,(1 + |v|)^b \leq c < \infty$
	with $b > 1$, then $\eta(\hat f, P) \leq \frac{2c}{b-1}\,(1 + P/2)^{1-b}$, and with
	Lemma~\ref{Lemma:polydecay} the estimate \eqref{eq:inverr} becomes
	\begin{eqnarray}
		\nonumber
		\max_{|x| \leq L/2} \bigl|f(x) - f_{M,P,L}(x)\bigr| &\leq&
		\frac{2c}{b-1}\left(1 + \frac{P}{2}\right)^{1-b}
		+ 2c\,(2^b - 1)\,\zeta(b)\,P^{1-b}
		+ P L\,\delta(f, L)\\
		&+&\label{eq:inverrpoly}  \frac{2\sqrt{1 + \pi^2/4}}{M}\,\|f\|_{W(C(\mathbb{R}),\ell^1(\mathbb{Z}))}\,.
		\end{eqnarray}
	For compactly supported $f$, where $\delta(f, L) = 0$ for sufficiently large $L$,
	the spatial error therefore decays as $\mathcal{O}(P^{1-b}) + \mathcal{O}(M^{-1})$,
	and balancing the two contributions suggests the choice $M \gtrsim P^{\,b-1}$. We
	remark that the leading $M^{-1}$-coefficient is in fact governed by
	$|s_{P,L}(P/2)| \approx |\hat f(P/2)|$, which is already small at the band edge, so
	that in practice the frequency discretization is markedly less critical than the
	truncation with respect to $P$.
\end{remark}
\medskip

\begin{remark}
	\label{Rem:smoothness}
	Roughly speaking, smoothness of $f$ implies polynomial decay of $\hat{f}$,
	and smoothness of $\hat{f}$ implies polynomial decay of $f$. If each
	derivative $f^{(\ell)}$, $\ell = 0,\ldots,b$, belongs to $L^1(\mathbb{R})$,
	then the Fourier transform of $f^{(\ell)}$ equals
	$(2\pi v\,{\mathrm{i}})^{\ell}\,\hat{f}(v)$ and is bounded on $\mathbb{R}$,
	so that
	$$
	\sup_{v \in \mathbb{R}} |\hat{f}(v)|\,(1 + |v|)^b \leq c < \infty\,.
	$$
	Conversely, if each derivative $\hat{f}^{(\ell)}$, $\ell = 0,\ldots,b$,
	belongs to $L^1(\mathbb{R})$, then the inverse Fourier transform of
	$\hat{f}^{(\ell)}$ equals $(-2\pi{\mathrm{i}}\,x)^{\ell}\,f(x)$ and is
	bounded on $\mathbb{R}$, so that
	$$
	\sup_{x \in \mathbb{R}} |f(x)|\,(1 + |x|)^b \leq c < \infty\,.
	$$
\end{remark}
\medskip

Now we estimate the decay rate \eqref{eq:decayrate} of a function $f \in C(\mathbb R)$ with polynomial decay.

\begin{lemma}
	\label{Lemma:polydecay}
	Let $L \in 2\mathbb{N}$ be given. Assume that $f \in C(\mathbb{R})$
	has polynomial decay
	$$
	\sup_{x \in \mathbb{R}} |f(x)|\,(1 + |x|)^a \leq c < \infty\,,
	\quad a > 1\,.
	$$
	Then the decay rate of $f$ with respect step size $L$ can be estimated by
	\begin{equation}
		\label{eq:deltapoly}
		\delta(f,L) \leq 2c\,(2^a - 1)\,\zeta(a)\,L^{-a}\,,
	\end{equation}
	where $\zeta(a)$ denotes the \emph{Riemann zeta function}
	$$
	\zeta(a) := \sum_{n=0}^{\infty} (n+1)^{-a} < 1 + \frac{1}{a-1}\,,
	\quad a > 1\,.
	$$
Further $f$ belongs to $W(C(\R),\ell^1(\Z))$ with the norm
$$
\| f\|_{W(C(\R),\ell^1(\Z))} \leq 2c\,\zeta(a) < \infty\,.
$$
\end{lemma}

\emph{Proof}.
By the polynomial decay of $f$, for all $k \in \mathbb{Z} \setminus \{0\}$
and $|x| \leq L/2$ it holds that
$$
\bigl|f(x + kL)\bigr| \leq c\,\bigl(1 + |x + kL|\bigr)^{-a}
\leq c\,\bigl(1 + \bigl(|k| - \tfrac{1}{2}\bigr)L\bigr)^{-a}\,.
$$
Hence it follows that
\begin{eqnarray*}
	\delta(f,L) &=& \max_{|x|\leq L/2} \sum_{k \in \mathbb{Z} \setminus \{0\}} \bigl|f(x + k L)\bigr| \\
	&\leq& 2c\sum_{k=1}^{\infty}
	\bigl(1 + \bigl(k - \tfrac{1}{2}\bigr)L\bigr)^{-a} = 2c\,L^{-a}\sum_{k=1}^{\infty} \bigl(k - \tfrac{1}{2} + \tfrac{1}{L}\bigr)^{-a} \\
	&<& 2c\,L^{-a}\sum_{n=0}^{\infty} \bigl(n + \tfrac{1}{2}\bigr)^{-a}
	= 2c\,\zeta\!\left(a, \tfrac{1}{2}\right)\,L^{-a}\,,
\end{eqnarray*}
where $\zeta\!\left(a, \frac{1}{2}\right) := \sum_{n=0}^{\infty}
\bigl(n + \frac{1}{2}\bigr)^{-a}$ denotes the \emph{Hurwitz zeta function}.
By $\zeta\!\left(a, \frac{1}{2}\right) = 2^a\sum_{n=0}^{\infty}(2n+1)^{-a}$
we obtain
\begin{eqnarray*}
	\zeta(a) &=& \sum_{n=0}^{\infty}(2n+1)^{-a} + \sum_{n=0}^{\infty}(2n+2)^{-a}
	= \sum_{n=0}^{\infty}(2n+1)^{-a} + 2^{-a}\sum_{n=0}^{\infty}(n+1)^{-a} \\
	&=& 2^{-a}\,\zeta\!\left(a,\tfrac{1}{2}\right) + 2^{-a}\,\zeta(a)
\end{eqnarray*}
such that $\zeta\!\left(a,\frac{1}{2}\right) = (2^a - 1)\,\zeta(a)$, which
gives \eqref{eq:deltapoly}. For $a \in 2\mathbb{N}$, explicit values of
$\zeta(a)$ are known, namely $\zeta(a) = \frac{|B_a|}{2\,a!}\,(2\pi)^a$ with
the \emph{Bernoulli number} $B_a$ (see \cite[pp.~266--267]{Ap76}). The
integral test for convergence of series provides that for $a > 1$,
$$
\zeta(a) = \sum_{n=0}^{\infty}(n+1)^{-a} \leq 1 + \int_0^{\infty}(x+1)^{-a}\,
{\mathrm{d}}x = 1 + \frac{1}{a-1}\,.
$$
Obviously, we have $f \in W(C(\R),\ell^1(\Z))$, since it holds by assumption that
$$
\| f\|_{W(C(\R),\ell^1(\Z))} \leq c \sum_{n=0}^{\infty} (1+n)^{-a} + c \sum_{n=1}^{\infty}\big(1 + (n-1)\big)^{-a} = 2c\,\zeta(a) < \infty\,.
$$
This completes the proof. \qedsymbol
\medskip

Using the corresponding decay rates of $f$ and $\hat f$ which have both polynomial decays, we can estimate the scaled uniform maximum error $M_{P,L}(f)$:

\begin{theorem}
	\label{Thm:polydecay}
	Let $L,\, P \in 2\mathbb{N}$ be given. If both functions
	$f,\, \hat{f} \in C(\mathbb{R})$ have polynomial decays
	\begin{eqnarray*}
		\sup_{x \in \mathbb{R}} |f(x)|\,(1 + |x|)^a &\leq& c < \infty\,,
		\quad a > 1\,,\\
		\sup_{v \in \mathbb{R}} |\hat{f}(v)|\,(1 + |v|)^b &\leq& d < \infty\,,
		\quad b > 1\,,
	\end{eqnarray*}
	then it holds the error estimate
	\begin{equation}
		\label{eq:maxerrorpolydecay}
		M_{P,L}^{\ast}(f) \leq M_{P,L}(f) \leq 2d\,(2^b - 1)\,\zeta(b)\,L^{-1/2}\,P^{-b}
		+ 2c\,(2^a - 1)\,\zeta(a)\,L^{-a + 1/2}\,.
	\end{equation}
\end{theorem}

\emph{Proof}. Using Theorem~\ref{Thm:samplingpolynomials}, it follows that
$$
M_{P,L}^{\ast}(f) \leq M_{P,L}(f) \leq \frac{1}{\sqrt L}\,\delta(\hat f, P) + \sqrt L\,\delta(f,L)\,.
$$
Then Lemma~\ref{Lemma:polydecay} provides that
\begin{eqnarray*}
	\delta(\hat{f},P) &\leq& 2d\,(2^b - 1)\,\zeta(b)\,P^{-b}\,,\\
	\delta(f, L) &\leq& 2c\,(2^a - 1)\,\zeta(a)\,L^{-a}\,.
\end{eqnarray*}
This completes the proof. \qedsymbol
\medskip

\begin{remark}
	We compare Theorem~\ref{Thm:polydecay} with the corresponding result of
	\cite[Theorem~2.1]{EhGrKl24}. From \cite[Theorem~2.1]{EhGrKl24} it follows
	that
	$$
	M_{P,L}^{\ast}(f) \leq E_{P,L}^{\ast}(f) \leq C\,(L^{-\alpha} + P^{-\beta})
	$$
	with arbitrary $\alpha \in \big(0,\, a - \frac{1}{2}\big)$,
	$\beta \in \big(0,\, b - \frac{1}{2}\big)$, and some constant $C > 0$.
	Our estimate \eqref{eq:maxerrorpolydecay} is explicit and more practicable,
	since all constants are described directly by the assumptions of
	Theorem~\ref{Thm:polydecay}\new{, and it does not lose an arbitrarily small
	$\varepsilon>0$ in the exponents}. Moreover, using Theorem~\ref{Thm:polydecay},
	the scaled Euclidean approximation error can be bounded by
	$$
	E_{P,L}^{\ast}(f) \leq \sqrt{LP}\,M_{P,L}^{\ast}(f)
	\leq 2d\,(2^b - 1)\,\zeta(b)\,\new{P^{-b+1/2}}
		+ 2c\,(2^a - 1)\,\zeta(a)\,L^{1 -a}\,P^{1/2}\,.
	$$
\end{remark}
\medskip

\new{
The estimate \eqref{eq:maxerrorpolydecay} is not symmetric in the two
parameters: the decay of $\hat f$ enters with the full exponent $b$, whereas
the decay of $f$ enters with the reduced exponent $a - \frac12$. Since the
corresponding bound of \cite[Theorem~2.1]{EhGrKl24} \emph{is} symmetric, we
explain in the following two remarks where this difference comes from. It turns
out that it is caused neither by the proof technique nor by the normalization,
but by a genuine structural difference between the two error contributions
$\sigma_1$ and $\sigma_2$ in the proof of
Theorem~\ref{Thm:samplingpolynomials}.

\begin{remark}
\label{Rem:symmetry}
Let $f$, $\hat f$ have polynomial decay with exponents $a$, $b$ as in
Theorem~\ref{Thm:polydecay}. Consider the two remainders
$$
\sigma_1(v) = \sum_{k \in \mathbb{Z}\setminus\{0\}} \hat f(v + kP)\,,
\qquad
\sigma_2(v) = \frac{1}{P} \sum_{n \in \mathbb{Z}\setminus[LP]}
f\!\left(\frac{n}{P}\right){\mathrm e}^{-2\pi{\mathrm i}\,nv/P}
$$
separately, and compare their maximum norm on $[-P/2,\,P/2]$ with their
quadratic mean
$\bigl(\frac{1}{P}\int_{-P/2}^{P/2}|\sigma_j(v)|^2\,{\mathrm d}v\bigr)^{1/2}$.

The \emph{aliasing} remainder $\sigma_1$ is a periodization of $\hat f$. It is
non-oscillatory, and both its maximum and its quadratic mean are of the exact
order $P^{-b}$; the quotient of the two is bounded above and below by absolute
constants.

The \emph{truncation} remainder $\sigma_2$ behaves completely differently. By
the orthogonality of the exponentials ${\mathrm e}^{-2\pi{\mathrm i}\,nv/P}$ on
$[-P/2,\,P/2]$ we have exactly
$$
\frac{1}{P}\int_{-P/2}^{P/2} |\sigma_2(v)|^2\,{\mathrm d}v
= \frac{1}{P^2} \sum_{n \in \mathbb{Z}\setminus[LP]}
\Bigl|f\!\left(\frac{n}{P}\right)\Bigr|^2
= \Theta\bigl(P^{-1}\,L^{1-2a}\bigr)\,,
$$
so that the quadratic mean of $\sigma_2$ is of the order
$P^{-1/2}\,L^{1/2-a}$, whereas for nonnegative $f$ at $v = 0$ all summands add
up constructively and give
$$
\max_{|v| \le P/2}|\sigma_2(v)| \ \ge\ |\sigma_2(0)|
= \frac{1}{P} \sum_{n \in \mathbb{Z}\setminus[LP]}
f\!\left(\frac{n}{P}\right)
= \Theta\bigl(L^{1-a}\bigr)\,.
$$
Here the two-sided order estimates presuppose a matching two-sided polynomial
decay $c_0\,(1+|x|)^{-a} \le |f(x)| \le c\,(1+|x|)^{-a}$; the lower estimate for
$|\sigma_2(0)|$ is exactly the one established in Theorem~\ref{Thm:lowerbound}.
Hence the maximum of $\sigma_2$ exceeds its quadratic mean by the factor
$\sqrt{LP}$, while for $\sigma_1$ the two agree up to constants.

This is the source of the asymmetry. Passing from an averaged error to a
uniform error costs nothing in the aliasing part but costs the full factor
$\sqrt{LP}$ in the truncation part, and no symmetric bound for the uniform
error can therefore exist. We add that after the normalization by $L^{-1/2}$
both error measures predict \emph{the same} rate in the truncation parameter,
namely $L^{1/2-a}$; the difference between \eqref{eq:maxerrorpolydecay} and
\cite[Theorem~2.1]{EhGrKl24} is confined to the aliasing parameter, where our
bound is better by the factor $\sqrt{L P}$. This is confirmed numerically in
Example~\ref{Ex:Ldependence} and Figure~\ref{fig:EstarM}.
\end{remark}
\medskip

\begin{remark}
\label{Rem:notequivalent}
The two error measures $M_{P,L}^{\ast}(f)$ and $E_{P,L}^{\ast}(f)$ are
\emph{not} equivalent. Indeed, for any $f$ we have the elementary two-sided
estimate
\begin{equation}
\label{eq:MstarEstar}
1 \;\le\; \frac{E_{P,L}^{\ast}(f)}{M_{P,L}^{\ast}(f)} \;\le\; \sqrt{LP}\,,
\end{equation}
since a maximum of $LP$ numbers is at most their Euclidean norm and at least
$(LP)^{-1/2}$ times that norm. Both bounds in \eqref{eq:MstarEstar} are
essentially attained:
\begin{itemize}
\item If the error is \emph{truncation dominated}, i.e.\ if
$\delta(\hat f, P) \ll L\,\delta(f,L)$, then by
Remark~\ref{Rem:symmetry} the quotient stays bounded. For the bandlimited
function $g$ of Example~\ref{Ex:Ldependence} one finds numerically
$E_{P,L}^{\ast}(g)/M_{P,L}^{\ast}(g) \to 1.244$ as $L \to \infty$,
independently of $P$.
\item If the error is \emph{aliasing dominated}, i.e.\ if
$L\,\delta(f,L) \ll \delta(\hat f,P)$, the quotient grows like $\sqrt{LP}$.
For the compactly supported B-spline $M_2$ with $L = 2$ one measures
$E_{P,L}^{\ast}/M_{P,L}^{\ast} = 22.18$ at $P = 1024$, while
$\sqrt{LP} = 45.25$; see Figure~\ref{fig:EstarM}.
\end{itemize}
Consequently, an estimate for $E_{P,L}^{\ast}(f)$ may overestimate
$M_{P,L}^{\ast}(f)$ by a factor which grows without bound, and
Theorem~\ref{Thm:polydecay} cannot be recovered from
\cite[Theorem~2.1]{EhGrKl24}. We stress in addition that
$M_{P,L}(f)$, which controls the error at \emph{every} $v$ with
$|v| \le P/2$, is not comparable to $E_{P,L}^{\ast}(f)$ at all, since the
latter contains no information about the values of $\hat f - s_{P,L}$ off the
grid $\{k/L : k \in [LP]\}$.
\end{remark}
\medskip
}

As known, a nontrivial function cannot be bandlimited and space-limited (see \cite[pp.~103--104]{PlPoStTa23}).
If $f$ is space-limited or bandlimited, then we obtain simpler error estimates:

\begin{corollary}
	\label{Cor:bandlimited}
	Let $L,\, P \in 2\mathbb{N}$ be given. If $f,\, \hat{f} \in C(\mathbb{R})$
	satisfy
	\begin{eqnarray*}
		\sup_{x \in \mathbb{R}} |f(x)|\,(1 + |x|)^a &\leq& c < \infty\,,
		\quad a > 1\,,\\
		\mathrm{supp}\,\hat{f} &\subseteq&
		\bigl[-\tfrac{M}{2},\,\tfrac{M}{2}\bigr]\,,
		\quad 0 < M \leq P\,,
	\end{eqnarray*}
	then it holds the error estimate
	$$
	M_{P,L}^{\ast}(f) \leq M_{P,L}(f) \leq 2c\,(2^a - 1)\,\zeta(a)\,L^{1/2-a}\,.
	$$
	If $f,\, \hat{f} \in C(\mathbb{R})$ satisfy
	\begin{eqnarray*}
		\mathrm{supp}\,f &\subseteq&
		\bigl[-\tfrac{M}{2},\,\tfrac{M}{2}\bigr]\,,
		\quad 0 < M \leq L\,,\\
		\sup_{v \in \mathbb{R}} |\hat{f}(v)|\,(1 + |v|)^b &\leq& d < \infty\,,
		\quad b > 1\,,
	\end{eqnarray*}
	then it holds the error estimate
	$$
	M_{P,L}^{\ast}(f) \leq M_{P,L}(f) \leq 2d\,(2^b - 1)\,\zeta(b)\,L^{-1/2}\,P^{-b}\,.
	$$
\end{corollary}

\emph{Proof}. This result follows immediately from Theorem \ref{Thm:polydecay}. If $\mathrm{supp}\,\hat{f} \subseteq
\bigl[-\frac{M}{2},\,\frac{M}{2}\bigr]$ and $M \leq P$, then
$\hat{f}(v + kP) = 0$ for all $k \in \mathbb{Z} \setminus \{0\}$ and
$|v| \leq P/2$, so that $\delta(\hat{f}, P) = 0$.
If $\mathrm{supp}\,f \subseteq \bigl[-\frac{M}{2},\,\frac{M}{2}\bigr]$ and
$M \leq L$, then all samples $f\!\left(\frac{n}{P}\right)$ vanish for
$n \in \mathbb{Z} \setminus [LP]$, so that $\delta(f, L) = 0$. \qedsymbol
\medskip

Now we estimate the decay rate \eqref{eq:decayrate} of a function with exponential decay.

\begin{lemma}
	\label{Lemma:expodecay}
	Let $L \in 2\mathbb{N}$ be given. Let $f \in C(\mathbb{R})$ be a
	function with exponential decay
	$$
	\sup_{x \in \mathbb{R}} |f(x)|\,{\mathrm{e}}^{r\,|x|^{\alpha}}
	\leq c < \infty\,, \quad r,\, \alpha > 0\,.
	$$
	Assume that in the case $0 < \alpha < 1$ it holds $r\,(L/2)^{\alpha} > \frac{1}{\alpha} - 1$.\\
Then the decay rate of $f$ with respect to step size $L$ can be
	estimated by
	\begin{equation}
		\label{eq:deltaexpo}
		\delta(f,L) \leq 2c\,\gamma(r, \alpha,L) \,{\mathrm{e}}^{-r\,(L/2)^{\alpha}}
	\end{equation}
with
$$
\gamma(r,\alpha,L) := \left\{\begin{array}{ll}
1 + \frac{r^{1/\alpha}}{L\,(\alpha r (L/2)^{\alpha} - 1 + \alpha)} & \quad 0< \alpha <1\,, \\
1 + \frac{1}{r L} &\quad \alpha = 1\,,\\
1 + \frac{2\,(r (L/2)^{\alpha} + b_{\alpha})^{1/\alpha} - r^{1/\alpha} L}{4 b_{\alpha}} & \quad \alpha>1\,.
\end{array}\right.
$$
For $\alpha > 1$, it holds $b_{\alpha} := \Gamma\big(\frac{1}{\alpha} +1\big)^{\alpha/(1- \alpha)}$. Further $f$ belongs to $W(C(\R), \ell^1(\Z))$ and has the norm
$$
\| f \|_{W(C(\R),\ell^1(\Z))} \leq 2c + \tfrac{2c}{\alpha}\,r^{-1/\alpha}\, \Gamma\big(\tfrac{1}{\alpha}\big) < \infty\,.
$$
\end{lemma}

\emph{Proof}.
By the exponential decay of $f$, for all $k \in \mathbb{Z} \setminus \{0\}$
and $|x| \leq L/2$ it holds that
$$
\bigl|f(x + k L)\bigr| \leq c\,{\mathrm{e}}^{-r\,|x + kL|^{\alpha}}
\leq c\,{\mathrm{e}}^{-r\,(L/2)^{\alpha}\,(2|k| - 1)^{\alpha}}\,.
$$
Hence it follows that
\begin{eqnarray*}
	\delta(f,L) &=& \max_{|x| \leq L/2} \sum_{k \in \mathbb{Z} \setminus \{0\}} \bigl|f(x + k L)\bigr|
	\leq 2c\sum_{k=1}^{\infty}
	{\mathrm{e}}^{-r\,(L/2)^{\alpha}\,(2k-1)^{\alpha}}\\
	&=& 2c \sum_{n=0}^{\infty}
	{\mathrm{e}}^{-r\,(L/2)^{\alpha}\,(2n+1)^{\alpha}}\,.
\end{eqnarray*}
Setting $u := r\,(L/2)^{\alpha}$ for shortness, the integral test for convergence of
series yields
\begin{eqnarray*}
	\sum_{n=0}^{\infty} {\mathrm{e}}^{-u\,(2n+1)^{\alpha}}
	&\leq& {\mathrm{e}}^{-u} + \int_0^{\infty}
	{\mathrm{e}}^{-u\,(2x+1)^{\alpha}}\,{\mathrm{d}}x
	= {\mathrm{e}}^{-u} + \frac{1}{2\alpha}\,u^{-1/\alpha}\,
	\Gamma\!\left(\tfrac{1}{\alpha}, u\right),
\end{eqnarray*}
where
$$
\Gamma\big(\frac{1}{\alpha},u\big) := \int_u^{\infty} t^{1/\alpha - 1}\,{\mathrm e}^{-t}\,{\mathrm d}t
$$
denotes the \emph{upper incomplete gamma function}. A simple check shows that
$$
- \frac{1}{\alpha}\,u^{-1/\alpha}\,\Gamma\big(\tfrac{1}{\alpha}, t^{\alpha}\,u\big)
$$
is a primitive of ${\mathrm e}^{-u\,t^{\alpha}}$ and therefore
$$
\int_1^{\infty} {\mathrm e}^{-u\,t^{\alpha}}\,{\mathrm d}t = \frac{1}{\alpha}\,u^{-1/\alpha}\,\Gamma\big(\tfrac{1}{\alpha}, u\big)\,.
$$
By \cite[Theorem 1.1 and Proposition 2.7]{Pi20}, an upper bound of the upper incomplete gamma function $\Gamma\big(\tfrac{1}{\alpha}, u\big)$ is given by
$$
\Gamma\big(\tfrac{1}{\alpha}, u\big) \leq \left\{ \begin{array}{ll}
\frac{u^{1/\alpha}}{u + 1-1/\alpha}\,{\mathrm e}^{-u} & \quad 0 < \alpha < 1\,,\\
\big(1 + \frac{1}{2u}\big)\,{\mathrm e}^{-u} & \quad \alpha = 1\,,\\
\frac{\alpha\,(u + b_{\alpha})^{1/\alpha} - u^{1/\alpha}}{b_{\alpha}}\,{\mathrm e}^{-u} & \quad \alpha > 1\,.
\end{array}\right.
$$
This implies the estimate \eqref{eq:deltaexpo}.
Obviously, we have $f \in W(C(\R), \ell^1(\Z))$, since it holds by assumption and \cite[Formula 3.478]{GR82} that
\begin{eqnarray*}
\| f \|_{W(C(\R),\ell^1(\Z))} &\leq& c \sum_{n=0}^{\infty} {\mathrm e}^{-r\,n^{\alpha}} + c \sum_{n=1}^{\infty} {\mathrm e}^{-r\,(n-1)^{\alpha}} = 2 c \sum_{n=0}^{\infty} {\mathrm e}^{-r\,n^{\alpha}}\\
&\leq& 2c + 2c \int_0^{\infty}{\mathrm e}^{-r\,t^{\alpha}}\,{\mathrm d}t = 2c + \tfrac{2c}{\alpha}\,r^{-1/\alpha}\, \Gamma\big(\tfrac{1}{\alpha}\big) < \infty\,.
\end{eqnarray*}
This completes the proof. \qedsymbol
\medskip

Using the decay rates of $f$ and $\hat f$ which have both exponential decays, we can estimate the scaled uniform maximum error $M_{P,L}(f)$:

\begin{theorem}
	\label{Thm:expodecay}
	Let $L,\, P \in 2\mathbb{N}$ be given. Let $f,\, \hat{f} \in C(\mathbb{R})$
	be functions with exponential decay
	\begin{eqnarray*}
		\sup_{x \in \mathbb{R}} |f(x)|\,{\mathrm{e}}^{r\,|x|^{\alpha}}
		&\leq& c < \infty\,, \quad r,\,\alpha > 0\,,\\
		\sup_{v \in \mathbb{R}} |\hat{f}(v)|\,{\mathrm{e}}^{s\,|v|^{\beta}}
		&\leq& d < \infty\,, \quad s,\,\beta > 0\,.
	\end{eqnarray*}
	Assume that $r\,(L/2)^{\alpha} > \frac{1}{\alpha} - 1$ in the case $0 < \alpha <1$ and that $s\,(P/2)^{\beta}> \frac{1}{\beta} - 1$ in the case $0 < \beta < 1$.\\
	Then it holds the error estimate
	\begin{equation}
		\label{eq:maxerrorexpodecay}
		M_{P,L}(f) \leq
		2d\,\gamma(s,\beta,P)\, L^{-1/2}\,{\mathrm{e}}^{-s\,(P/2)^{\beta}} + 2c\,\gamma(r,\alpha,L)\,L^{1/2}\,{\mathrm{e}}^{-r\,(L/2)^{\alpha}}\,,
	\end{equation}
where the constants $\gamma(s,\beta,P)$ and $\gamma(r,\alpha,L)$ are defined in Lemma \ref{Lemma:expodecay}.
\end{theorem}

\emph{Proof}. Using Theorem~\ref{Thm:samplingpolynomials}, it follows that
$$
M_{P,L}(f)
\leq \frac{1}{\sqrt L}\,\delta(\hat f, P) + \sqrt L\, \delta(f,L)\,.
$$
Then Lemma~\ref{Lemma:expodecay} provides that
\begin{eqnarray*}
	\delta(\hat{f}, P)&\leq& 2d\,\gamma(s,\beta,P)\,{\mathrm{e}}^{-s\,(P/2)^{\beta}}\,,\\
	\delta(f, L)&\leq& 2c\,\gamma(r,\alpha,L)\,
	{\mathrm{e}}^{-r\,(L/2)^{\alpha}}\,.
\end{eqnarray*}
This completes the proof. \qedsymbol
\medskip

\begin{remark}
\label{Rem:errorexpodecay}
We compare Theorem \ref{Thm:expodecay} with the corresponding result in \cite[Theorem 2.3]{EhGrKl24}. Then from \cite[Theorem 2.3]{EhGrKl24} it follows immediately that
$$
M_{P,L}^{\ast}(f) \leq E_{P,L}^{\ast}(f) \leq C\,\big( {\mathrm e}^{-r' (L/2)^{\alpha}} + {\mathrm e}^{-s' (P/2)^{\beta}}\big)
$$
with arbitrary $r' \in (0,\,r)$, $s' \in (0,\,s)$, and certain constant $C>0$. Note that in \cite[Theorem 2.3]{EhGrKl24} it is considered only the case $0 < \alpha,\,\beta \leq 1$.
Thus we see that our estimate \eqref{eq:maxerrorexpodecay} is more general, explicit, and more practicable, since all constants in \eqref{eq:maxerrorexpodecay} are known from
the decay conditions of Theorem \ref{Thm:expodecay}.
\end{remark}
\medskip

Using the decay rates of $f$ and $\hat f$ which have mixed decays, we can estimate the scaled uniform maximum error $M_{P,L}(f)$:

\begin{theorem}
	\label{Thm:mixeddecay}
	Let $L,\, P \in 2\mathbb{N}$ be given. Let $f,\, \hat{f} \in C(\mathbb{R})$
	be functions with mixed decay
	\begin{eqnarray*}
		\sup_{x \in \mathbb{R}} |f(x)|\,(1 + |x|)^a
		&\leq& c < \infty\,, \quad a > 1\,,\\
		\sup_{v \in \mathbb{R}} |\hat{f}(v)|\,{\mathrm{e}}^{s\,|v|^{\beta}}
		&\leq& d < \infty\,, \quad s,\,\beta > 0\,.
	\end{eqnarray*}
	Assume that $s\,(P/2)^{\beta} > \frac{1}{\beta} - 1$ in the case $0 < \beta < 1$. Then it holds the error estimate
	$$
	M_{P,L}(f) \leq
	2d\,\gamma(s,\beta,P)\,L^{-1/2}\,
	{\mathrm{e}}^{-s\,(P/2)^{\beta}}
	+ 2c\,(2^a - 1)\,\zeta(a)\,L^{-a + 1/2}\,.
	$$
	Alternatively, let $f,\, \hat{f} \in C(\mathbb{R})$ be functions with
	mixed decay
	\begin{eqnarray*}
		\sup_{x \in \mathbb{R}} |f(x)|\,{\mathrm{e}}^{r\,|x|^{\alpha}}
		&\leq& c < \infty\,, \quad r,\,\alpha > 0\,,\\
		\sup_{v \in \mathbb{R}} |\hat{f}(v)|\,(1 + |v|)^b
		&\leq& d < \infty\,, \quad b > 1\,.
	\end{eqnarray*}
	Assume that $r\,(L/2)^{\alpha} > \frac{1}{\alpha} - 1$ in the case $0 < \alpha < 1$. Then it holds the error estimate
	\begin{equation}
		\label{eq:maxerrormixeddecay}
		M_{P,L}(f) \leq
		2d\,(2^b - 1)\,\zeta(b)\,L^{-1/2}\,P^{-b}
		+ 2c\,\gamma(r,\alpha,L)\,L^{1/2}\,
		{\mathrm{e}}^{-r\,(L/2)^{\alpha}}\,.
	\end{equation}
\end{theorem}

\emph{Proof}. In both cases, Theorem~\ref{Thm:samplingpolynomials} yields
$$
M_{P,L}(f)
\leq \frac{1}{\sqrt L}\,\delta(\hat{f}, P) + \sqrt L\, \delta(f, L)\,.
$$
Then the assertion follows immediately from Lemma~\ref{Lemma:polydecay} and
Lemma~\ref{Lemma:expodecay}. \qedsymbol
\medskip

In practice, $L,\, P \in 2\mathbb{N}$ are sufficiently large, so that the
exponential error terms in the estimates of Theorem~\ref{Thm:mixeddecay} are
negligible. Thus $M_{P,L}(f)$ is dominated by the corresponding polynomial
error term.
\medskip

\begin{remark}
\label{Rem:errormixeddecay}
We compare the second case of Theorem \ref{Thm:mixeddecay} with the corresponding result in \cite[Theorem 2.4]{EhGrKl24}. Then from \cite[Theorem 2.4 and (4.21)]{EhGrKl24} it follows immediately that
$$
M_{P,L}^{\ast}(f) \leq E_{P,L}^{\ast}(f) \leq C\,\big( {\mathrm e}^{-r' (L/2)^{\alpha}} + P^{-\beta}\big)
$$
with arbitrary $r' \in (0,\,r)$, $\beta \in \big(0,\,b - \frac{1}{2}\big)$, and certain constant $C>0$. Note that in \cite[Theorem 2.4]{EhGrKl24} it is considered only the case $0 < \alpha\leq 1$.
Thus we see that our estimate \eqref{eq:maxerrormixeddecay} is more general, explicit, and more practicable, since all constants in \eqref{eq:maxerrormixeddecay} are known from
the decay conditions of Theorem \ref{Thm:mixeddecay}.
\end{remark}
\medskip

\new{
\subsection{A lower bound for the uniform approximation error}

All estimates obtained so far are upper bounds. We now show that the rate in
the truncation parameter $L$ predicted by Theorem~\ref{Thm:polydecay} is
sharp. This is in
particular relevant for Remark~\ref{Rem:symmetry}, where we claimed that the
asymmetry of \eqref{eq:maxerrorpolydecay} is a genuine effect.

\begin{theorem}
\label{Thm:lowerbound}
Let $L$, $P \in 2\mathbb{N}$ be given, and let
$f \in W(C(\mathbb{R}),\ell^1(\mathbb{Z}))$ with
$\hat f \in W(C(\mathbb{R}),\ell^1(\mathbb{Z}))$. Assume that $f$ is
nonnegative, that $f$ is non-increasing on $[x_0,\,\infty)$ for some
$x_0 \ge 0$, and that
$$
f(x) \ \ge\ c_0\,(1 + x)^{-a}\,, \qquad x \ge x_0\,,
$$
with constants $c_0 > 0$ and $a > 1$. Then for all $L \ge 2\,x_0$ it holds
\begin{equation}
\label{eq:lowerbound}
M_{P,L}(f)\ \ge\ M_{P,L}^{\ast}(f) \ \ge\ \frac{1}{\sqrt L}
\left(\frac{c_0}{a-1}\left(1 + \frac{L}{2} + \frac{1}{P}\right)^{1-a}
- \delta(\hat f, P)\right).
\end{equation}
The bound \eqref{eq:lowerbound} is informative in the truncation dominated
regime $\delta(\hat f,P) < \frac{c_0}{a-1}(1 + \frac{L}{2} + \frac{1}{P})^{1-a}$.
In particular, if $\hat f$ is bandlimited with
$\mathrm{supp}\,\hat f \subseteq [-\frac{M}{2},\,\frac{M}{2}]$ and $M \le P$,
then $\delta(\hat f,P) = 0$ and
$$
M_{P,L}(f) \ \ge\ \frac{c_0}{a-1}\,\Bigl(1 + \frac{L}{2} + \frac{1}{P}\Bigr)^{1-a}\,
L^{-1/2}\ =\ \Omega\bigl(L^{-a+1/2}\bigr) \qquad (L \to \infty)\,,
$$
uniformly in $P \ge M$. Together with the upper bound
$\mathcal{O}(L^{-a+1/2})$ of Theorem~\ref{Thm:polydecay} and
Corollary~\ref{Cor:bandlimited}, this yields the two-sided estimate
$M_{P,L}(f) = \Theta\bigl(L^{-a+1/2}\bigr)$.
\end{theorem}

\emph{Proof}. As in the proof of Theorem~\ref{Thm:samplingpolynomials} we have
$\hat f(v) - s_{P,L}(v) = \sigma_2(v) - \sigma_1(v)$ for all $v \in \mathbb{R}$.
Evaluating at the grid point $v = 0$, which belongs to
$\{k/L:\, k \in [LP]\}$, we obtain
$$
\sqrt{L}\; M_{P,L}^{\ast}(f) \ \ge\ \bigl|\hat f(0) - s_{P,L}(0)\bigr|
= \bigl|\sigma_2(0) - \sigma_1(0)\bigr|
\ \ge\ |\sigma_2(0)| - |\sigma_1(0)|\,.
$$
By the proof of Theorem~\ref{Thm:samplingpolynomials} we have
$|\sigma_1(0)| \le \delta(\hat f,P)$. Since $f \ge 0$, the value
$$
\sigma_2(0) = \frac{1}{P} \sum_{n \in \mathbb{Z}\setminus[LP]}
f\!\left(\frac{n}{P}\right)
$$
is nonnegative, so that $|\sigma_2(0)| = \sigma_2(0)$, and discarding the
indices $n \le -\frac{LP}{2}$ only decreases the sum,
$$
\sigma_2(0) \ \ge\ \frac{1}{P}\sum_{n > LP/2} f\!\left(\frac{n}{P}\right)
= \frac{1}{P}\sum_{n \ge LP/2 + 1} f\!\left(\frac{n}{P}\right).
$$
Since $L \ge 2 x_0$, the function $f$ is non-increasing on
$[\frac{L}{2},\,\infty)$, and the right-hand side is a lower Riemann sum with
step size $\frac 1P$ of the integral of $f$ over
$\bigl[\frac{L}{2} + \frac{1}{P},\,\infty\bigr)$. Hence
$$
\sigma_2(0) \ \ge\ \int_{L/2 + 1/P}^{\infty} f(x)\,{\mathrm d}x
\ \ge\ c_0 \int_{L/2 + 1/P}^{\infty} (1+x)^{-a}\,{\mathrm d}x
= \frac{c_0}{a-1}\left(1 + \frac{L}{2} + \frac{1}{P}\right)^{1-a}.
$$
Together with $M_{P,L}^{\ast}(f) \le M_{P,L}(f)$ this gives
\eqref{eq:lowerbound}. \qedsymbol
\medskip

\begin{remark}
Theorem~\ref{Thm:lowerbound} shows that the exponent $a - \frac12$ of $L$ in
\eqref{eq:maxerrorpolydecay} is optimal, and that the maximum of
$|\hat f - s_{P,L}|$ is, in the truncation dominated regime, attained near
$v=0$ and is of the order of the $L^1$-tail
$\int_{|x| > L/2} |f(x)|\,{\mathrm d}x$ of $f$. The assumptions of
Theorem~\ref{Thm:lowerbound} are satisfied, e.g., by the bandlimited function
$g$ of Example~\ref{Ex:Ldependence} and, more generally, by every nonnegative
even function with a monotone tail of exact polynomial order. They are not
satisfied by oscillating functions, for which cancellation in $\sigma_2(0)$ may
occur; in that case the upper bound of Theorem~\ref{Thm:polydecay} need not be
attained.
\end{remark}
\medskip

\subsection{The resulting algorithm\label{sec:algorithm}}

We now summarize the method. The input consists of the function $f$, the decay
data of $f$ and $\hat f$, a target accuracy $\varepsilon > 0$ and the nodes at
which $\hat f$ is to be evaluated. The output consists of approximate values of
$\hat f$ together with a rigorous a~priori error bound.

\begin{algorithm}{Computation of the Fourier transform via \NFFT}
\label{Alg:FTviaNFFT}
\textbf{Input:} $f \in W(C(\mathbb{R}),\ell^1(\mathbb{Z}))$ with
$\hat f \in W(C(\mathbb{R}),\ell^1(\mathbb{Z}))$, given as a routine which
evaluates $f$; decay parameters of $f$ and $\hat f$ as in
Theorem~\ref{Thm:polydecay}, Theorem~\ref{Thm:expodecay} or
Theorem~\ref{Thm:mixeddecay}; target accuracy $\varepsilon > 0$; evaluation
nodes $v_1,\ldots,v_K \in \mathbb{R}$.

\begin{enumerate}[label=\arabic*., leftmargin=*]
\item \emph{Choice of the bandwidth $P$.} Choose $P \in 2\mathbb{N}$ with
$P \ge 2\max_j |v_j|$, so that all evaluation nodes lie in $[-P/2,\,P/2]$, and
so large that the aliasing term satisfies
$\delta(\hat f, P) \le \frac{\varepsilon}{2}$. For polynomial decay of
$\hat f$ this means $2d\,(2^b-1)\,\zeta(b)\,P^{-b} \le \frac{\varepsilon}{2}$
by Lemma~\ref{Lemma:polydecay}; for exponential decay use
Lemma~\ref{Lemma:expodecay}.
\item \emph{Choice of the truncation parameter $L$.} Choose
$L \in 2\mathbb{N}$ so large that the truncation term satisfies
$L\,\delta(f,L) \le \frac{\varepsilon}{2}$, again by
Lemma~\ref{Lemma:polydecay} or Lemma~\ref{Lemma:expodecay}. By
Theorem~\ref{Thm:samplingpolynomials} the resulting sampling polynomial then
satisfies $\max_{|v| \le P/2}|\hat f(v) - s_{P,L}(v)| \le \varepsilon$.
\item \emph{Sampling.} Evaluate $f$ at the $LP$ equispaced nodes $n/P$,
$n \in [LP]$, and form the coefficients
$c_n := \frac{1}{P}\,f\!\left(\frac{n}{P}\right)$.
\item \emph{Choice of the \NFFT{} parameters.} Choose the oversampling factor
$\sigma \ge 1$ (usually $\sigma = 2$) and the window cut-off parameter
$m \in \mathbb{N}$ so that the \NFFT{} error is below $\varepsilon$; by
\cite{PoTa21a,PoTa21b} the \NFFT{} error decays exponentially in $m$, and in
practice $m \le 8$ suffices, see Table~\ref{tab:cutoff}.
\item \emph{Evaluation.} Compute
$s_{P,L}(v_j) = \sum_{n \in [LP]} c_n\,
{\mathrm e}^{-2\pi{\mathrm i}\,n (v_j/P)}$, $j = 1,\ldots,K$,
by one \NFFT{} of size $LP$ with the nodes $x_j := v_j/P \in [-\frac12,\,\frac12)$.
If the $v_j$ are the equispaced points $k/L$, $k \in [LP]$, replace the \NFFT{} by
a single \FFT{} of length $LP$.
\end{enumerate}

\textbf{Output:} approximate values $s_{P,L}(v_j) \approx \hat f(v_j)$,
$j = 1,\ldots,K$, with the guarantee
$\max_j |\hat f(v_j) - s_{P,L}(v_j)| \le \sqrt{L}\,M_{P,L}(f) \le \varepsilon$
up to the \NFFT{} error of step~4.

\textbf{Cost:} $\mathcal{O}(LP)$ evaluations of $f$ in step~3 and
$\mathcal{O}\bigl(\sigma L P \log(\sigma L P) + m\,K\bigr)$ arithmetic
operations in step~5.
\end{algorithm}
\medskip

\begin{remark}
\label{Rem:totalerror}
The total error of Algorithm~\ref{Alg:FTviaNFFT} splits into the approximation
error analysed in this paper and the error of the \NFFT{} itself,
$$
\bigl|\hat f(v_j) - \tilde s_{P,L}(v_j)\bigr|
\ \le\ \underbrace{\max_{|v| \le P/2}\bigl|\hat f(v) - s_{P,L}(v)\bigr|}
_{\le\ \sqrt L\,M_{P,L}(f),\ \text{Theorem~\ref{Thm:samplingpolynomials}}}
+\ \underbrace{\bigl|s_{P,L}(v_j) - \tilde s_{P,L}(v_j)\bigr|}
_{\text{\NFFT{} error, see \cite{PoTa21a,PoTa21b}}}\,,
$$
where $\tilde s_{P,L}$ denotes the value actually computed by the \NFFT{}. It is
essential here that the first term is bounded \emph{uniformly} in $v$, because
the nodes $v_j$ are arbitrary; a bound on the grid $\{k/L\}$ only, as provided
by \cite{EhGrKl24}, would not suffice. Since the \NFFT{} error decays
exponentially in the cut-off parameter $m$, the second term can always be made
negligible against the first at a cost which is linear in $m$; see
Table~\ref{tab:cutoff}.
\end{remark}
}

\section{Numerical examples\label{sec:numerics}}

In the previous section, we derived explicit error estimates for the uniform
approximation error $M_{P,L}(f)$ under polynomial, exponential, and mixed decay
conditions on $f$ and $\hat{f}$. We now illustrate these theoretical bounds by
means of concrete examples. For each example, we specify the relevant decay
parameters and state the resulting error estimate from
Section~\ref{sec:samplingpolynomials} explicitly before comparing it
numerically with the true approximation error $M_{P,L}(f)$.

\begin{example}
	\label{Ex:M2}
	Let $M_2$ be the centered linear B-spline given by
	$$
	M_2(x) := \left\{\begin{array}{ll}
		1 - |x| & x \in [-1,\,1]\,,\\
		0 & x \in \mathbb{R} \setminus [-1,\,1]\,.
	\end{array}\right.
	$$
	Then $f = M_2$ has the Fourier transform
	$$
	\hat{f}(v) = \int_{-1}^1 (1 - |x|)\,\cos(2\pi v x)\,{\mathrm{d}}x
	= \bigl(\mathrm{sinc}(\pi v)\bigr)^2\,, \quad v \in \mathbb{R}\,.
	$$
	Since $\mathrm{supp}\,f = [-1,\,1]$, the function $f \in C(\mathbb{R})$ is
	space-limited with $M = 2$, and $\hat{f} \in C(\mathbb{R})$ has polynomial
	decay with $b = 2$. If we estimate the even function ${\hat f}(v)\,(1 + |v|)^2$ separately on the intervals $[0,\,\tfrac{1}{2}]$ and $[\tfrac{1}{2},\,\infty)$, we see that $d =\tfrac{9}{4}$ is a possible choice. By
	Corollary~\ref{Cor:bandlimited}, for any $L \geq 2$ it holds
	$$
	M_{P,L}(f) \leq 2d\,(2^2 - 1)\,\zeta(2)\,L^{-1/2}\,P^{-2}
	= \frac{\pi^2}{\sqrt{L}}\,d\,P^{-2}\,.
	$$
	Hence the error decays like $P^{-2}$ as $P \to \infty$, which is confirmed
	by the numerical results in Figure \ref{fig:M2} (left).
	More generally, for the B-spline $M_{2m}$ of even order $2m$, the Fourier
	transform $\widehat{M_{2m}}$ has polynomial decay with $b = 2m$, and
	Corollary~\ref{Cor:bandlimited} yields the estimate
	$$
	M_{P,L}(f) \leq 2d\,(2^{2m} - 1)\,\zeta(2m)\,L^{-1/2}\,P^{-2m}\,.
	$$
	This $P^{-2m}$ decay is illustrated in Figure \ref{fig:M4} (right, $m=2$) and
	Figure \ref{fig:M6} ($m=3$).
	
	Alternatively, consider the bandlimited function
	$$
	g(x) := \frac{M}{2}\,\left(\mathrm{sinc}\!\left(\frac{M}{2}\,\pi x\right)\right)^2,
	\quad x \in \mathbb{R}\,,
	$$
	whose Fourier transform is $\hat{g}(v) = M_2\!\left(\frac{2}{M}\,v\right)$
	with $\mathrm{supp}\,\hat{g} \subseteq \bigl[-\frac{M}{2},\,\frac{M}{2}\bigr]$.
	Since $g$ has polynomial decay with $a = 2$, Corollary~\ref{Cor:bandlimited}
	gives, for $P \geq M$,
	$$
	M_{P,L}(g) \leq 2c\,(2^2 - 1)\,\zeta(2)\,L^{1/2-2}
	= \new{\pi^2\,c\,L^{-3/2}}\,.
	$$
	Hence the error decays like $L^{-3/2}$ as $L \to \infty$\new{, which is
	confirmed numerically in Example~\ref{Ex:Ldependence} and
	Figure~\ref{fig:Ldep}}.
\end{example}

\begin{figure}[ht]
	\centering
	\begin{subfigure}[t]{0.48\textwidth}
		\centering
		\begin{tikzpicture}
		\begin{loglogaxis}[FTerrstyle, log basis x=2,
			xlabel={$P$}, ylabel={$M_{P,L}(f)$}, legend pos=north east]
		\addplot[FTest] coordinates {
			(2,3.925611e+00) (4,9.814028e-01) (8,2.453507e-01) (16,6.133767e-02)
			(32,1.533442e-02) (64,3.833605e-03) (128,9.584012e-04) (256,2.396003e-04)
			(512,5.990007e-05) (1024,1.497502e-05)};
		\addplot[FTsim] coordinates {
			(2,6.748202e-02) (4,2.002465e-02) (8,5.647815e-03) (16,1.515243e-03)
			(32,3.939908e-04) (64,1.005330e-04) (128,2.539690e-05) (256,6.382778e-06)
			(512,1.599926e-06) (1024,4.005130e-07)};
		\legend{estimate, $M_{P,L}(f)$}
		\end{loglogaxis}
		\end{tikzpicture}
		\caption{B-Spline $M_2$, order 2 ($L=2$)}
		\label{fig:M2}
	\end{subfigure}
	\hfill
	\begin{subfigure}[t]{0.48\textwidth}
		\centering
		\begin{tikzpicture}
		\begin{loglogaxis}[FTerrstyle, log basis x=2,
			xlabel={$P$}, ylabel={$M_{P,L}(f)$}, legend pos=north east]
		\addplot[FTest] coordinates {
			(2,1.630149e+00) (4,1.018843e-01) (8,6.367771e-03) (16,3.979857e-04)
			(32,2.487411e-05) (64,1.554632e-06) (128,9.716448e-08) (256,6.072780e-09)
			(512,3.795487e-10)};
		\addplot[FTsim] coordinates {
			(2,1.280503e-03) (4,1.454116e-04) (8,1.323428e-05) (16,1.022636e-06)
			(32,7.159879e-08) (64,4.744578e-09) (128,3.055551e-10) (256,1.938735e-11)
			(512,1.220921e-12)};
		\legend{estimate, $M_{P,L}(f)$}
		\end{loglogaxis}
		\end{tikzpicture}
		\caption{B-Spline $M_4$, order 4 ($L=4$)}
		\label{fig:M4}
	\end{subfigure}
	\captionsetup{font=small}
	\caption{Estimate \eqref{eq:maxerrorpolydecay} (Corollary \ref{Cor:bandlimited}) and
		simulated error $M_{P,L}(f)$ for Example \ref{Ex:M2}: the B-spline $M_2$ of order 2
		with $L=2$ (left, decay $P^{-2}$) and $M_4$ of order 4 with $L=4$ (right, decay $P^{-4}$).}
\end{figure}
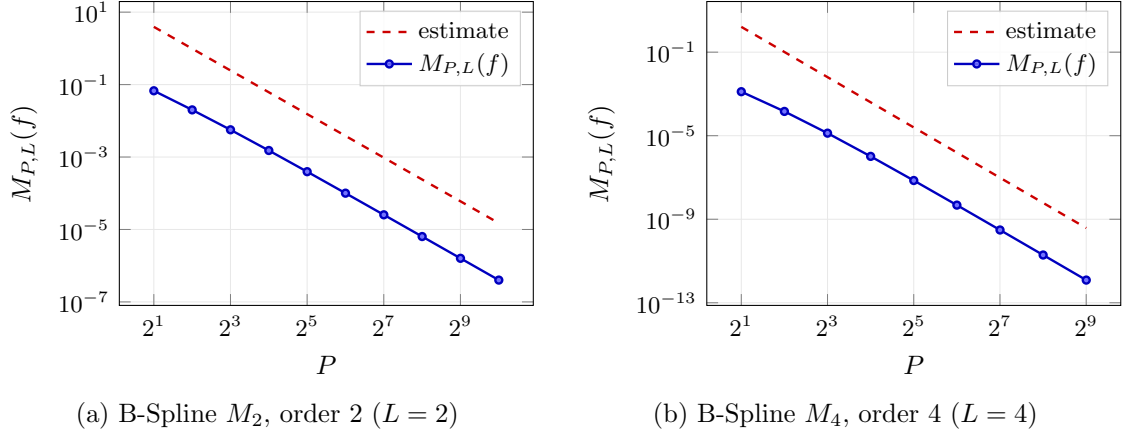

\begin{example}
	\label{Ex:Gaussian}
	The Gaussian function
	$$
	f(x) = {\mathrm{e}}^{-\pi x^2}\,, \quad x \in \mathbb{R}\,,
	$$
	has exponential decay with $r = \pi$ and $\alpha = 2$, and its Fourier
	transform $\hat{f}(v) = {\mathrm{e}}^{-\pi v^2}$ has the same exponential
	decay with $s = \pi$ and $\beta = 2$. Provided $\pi\,(L/2)^2 \geq 1$ and
	$\pi\,(P/2)^2 \geq 1$, Theorem~\ref{Thm:expodecay} yields
	\begin{align}
	M_{P,L}(f) &\leq \frac{1}{\sqrt{L}}\left(
	2\,\left(1 + \frac{1}{8}\right){\mathrm{e}}^{-\pi\,(P/2)^2}
	+ 2\,\left(1 + \frac{1}{8}\right)L\,{\mathrm{e}}^{-\pi\,(L/2)^2}\right)\\
	&= \frac{9}{4\sqrt{L}}\left({\mathrm{e}}^{-\pi P^2/4}
	+ L\,{\mathrm{e}}^{-\pi L^2/4}\right).
	\end{align}
	For fixed $L=P$ we obtain
	\[
	M_{L,L}(f)\le \frac{9(L+1)}{4\sqrt{L}} \, \mathrm{e}^{-\pi L^2/4}\,.
	\]	
	In Figure \ref{fig:Gaussian} we plot the error $M_{P,P}$ as well as the analytical estimate for $P=1,\ldots 8$.
\end{example}

% ===========================================================================
% Zweite Figure: M6 und Gaussian (rechte Seite mit deinen neuen Daten)
% ===========================================================================
\begin{figure}[ht]
	\centering
	\begin{subfigure}[t]{0.48\textwidth}
		\centering
		\begin{tikzpicture}
		\begin{loglogaxis}[FTerrstyle, log basis x=2,
			xlabel={$P$}, ylabel={$M_{P,L}(f)$}, legend pos=north east]
		\addplot[FTest] coordinates {
			(2,1.665065e+00) (4,2.601664e-02) (8,4.065099e-04) (16,6.351718e-06)
			(32,9.924559e-08) (64,1.550712e-09)};
		\addplot[FTsim] coordinates {
			(2,4.448083e-05) (4,1.847696e-06) (8,5.221655e-08) (16,1.135319e-09)
			(32,2.113998e-11) (64,3.615440e-13)};
		\legend{estimate, $M_{P,L}(f)$}
		\end{loglogaxis}
		\end{tikzpicture}
		\caption{B-Spline $M_6$, order 6 ($L=6$)}
		\label{fig:M6}
	\end{subfigure}
	\hfill
	\begin{subfigure}[t]{0.48\textwidth}
		\centering
		\begin{tikzpicture}
		\begin{semilogyaxis}[FTerrstyle,
			xlabel={$L=P$}, ylabel={$M_{P,P}(f)$}, legend pos=north east,
			xtick={1,2,3,4,5,6}, xmin=0.7, xmax=6.3]
		\addplot[FTest] coordinates {
			(1,2.051722e+00) (2,2.062588e-01) (3,4.424203e-03) (4,1.961630e-05)
			(5,1.792654e-08) (6,3.379230e-12)};
		\addplot[FTsim] coordinates {
			(1,5.440619e-01) (2,1.643293e-02) (3,5.526831e-04) (4,1.337982e-06)
			(5,1.431390e-09) (6,1.815531e-13)};
		\legend{estimate, $M_{P,P}(f)$}
		\end{semilogyaxis}
		\end{tikzpicture}
		\caption{Gaussian $e^{-\pi x^2}$ ($L=P$)}
		\label{fig:Gaussian}
	\end{subfigure}
	\captionsetup{font=small}
	\caption{Estimate and simulated error $M_{P,L}(f)$: the B-spline $M_6$ of order 6
		with $L=6$ for Example \ref{Ex:M2} (left, estimate \eqref{eq:maxerrorpolydecay},
		decay $P^{-6}$) and the Gaussian $\mathrm{e}^{-\pi x^2}$ with $L=P$ for
		Example \ref{Ex:Gaussian} (right, estimate $\tfrac{9(L+1)}{4\sqrt L}\,\mathrm{e}^{-\pi L^2/4}$).}
\end{figure}
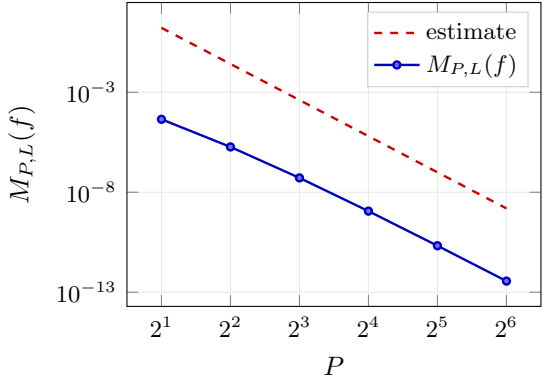
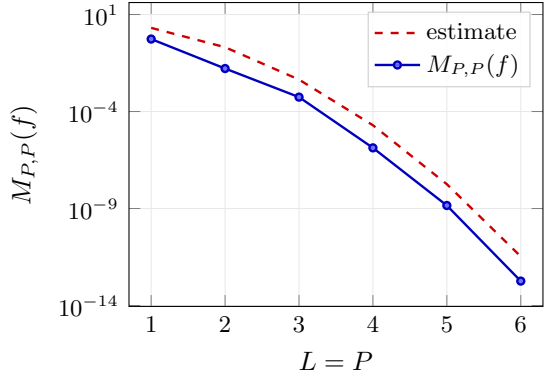

\begin{example}
	\label{ex:mixeddecay}
	\new{This example is taken from \cite[Section~3.2]{EhGrKl24}, where it is
	used to illustrate the mixed-decay case for the scaled Euclidean error;
	we use it here in order to make the two sets of results directly
	comparable, and additionally in Section~\ref{sec:cost}.}
	Let
	$$
	f(x) := {\mathrm{e}}^{-2\pi\,|x - \delta|}\,, \quad x \in \mathbb{R}\,,
	$$
	with shift parameter $\delta \in \mathbb{R}$. Its Fourier transform is
	$$
	\hat{f}(v) = \frac{{\mathrm{e}}^{-2\pi{\mathrm{i}}\,\delta v}}{\pi\,(1 + v^2)}\,,
	\quad v \in \mathbb{R}\,.
	$$
	Hence $f$ has exponential decay with $r = 2\pi$ and $\alpha = 1$, and
	$\hat{f}$ has polynomial decay with $b = 2$ and constant
	$d = 1/\pi$. This is the mixed-decay case of
	Theorem~\ref{Thm:mixeddecay}. Provided $2\pi\,(L/2) \geq 1$, i.e.,
	$L \geq 1/\pi$, we obtain
	$$
	M_{P,L}(f) \leq \frac{1}{\sqrt{L}}\left(
	\frac{2}{\pi}\,(2^2 - 1)\,\zeta(2)\,P^{-2}
	+ 2\,\left(1 + \frac{1}{2}\right)L\,{\mathrm{e}}^{-2\pi\,(L/2)}\right)
	= \frac{1}{\sqrt{L}}\left(\frac{2\pi}{P^2}
	+ 3L\,{\mathrm{e}}^{-\pi L}\right).
	$$
	For sufficiently large $L$, the exponential term is negligible, and the
	error is dominated by the polynomial term $\frac{2\pi}{\sqrt{L}\,P^2}$.
	A good parameter choice is $L = 8$ and $P = 2^t$, $t \geq 6$; for these
	values the exponential contribution is below machine precision and the
	estimate reduces to $M_{P,L}(f) \approx \frac{2\pi}{\sqrt{8}\,P^2}$,
	confirming the observed $P^{-2}$ decay in Figure \ref{fig:mixed_decay}.
\end{example}

\begin{example}\label{Ex:alg2.5}
	For fixed parameter $\beta \in {\mathbb N}$ we consider the  function
	$$
	f(x) := \left\{ \begin{array}{ll}
	(1-x^2)^{\beta-1/2} & \quad x \in [- 1,\,1]\,, \\
	0 & \quad x \in \mathbb R \setminus \big[-1,\,1]\,.
	\end{array}\right.
	$$
	Using  \cite[p.~8]{Oberh90}, we determine the corresponding Fourier transform
	\begin{eqnarray}
	\label{eq:FTvarphi0alg}
	\hat f(v)
	&=& \int_{-1}^1 (1-x^2)^{\beta-1/2}\,{\mathrm e}^{-2\pi {\mathrm i}\, v x}\, {\mathrm d}x = 2\,\int_0^1 (1-x^2)^{\beta-1/2}\,\cos(2\pi\, v x)\, {\mathrm d}x \nonumber \\ [1ex]
	&=& \frac{\pi \,(2\beta)!}{4^{\beta}\,\beta!} \left\{\begin{array}{ll}
	(\pi v)^{-\beta}\,J_{\beta}(2\pi  v)& \quad v \in {\mathbb R}\setminus \{0\}\,,\\
	\frac{1}{\beta!} & \quad v=0\,.
	\end{array} \right.
	\end{eqnarray}
	and obtain the estimate  \cite[Section 5.3]{PoTa21a}
	\begin{equation}
	\label{eq:FTvarphi0algestim}
	|\hat f(v)| \le \frac{3\,(2  \beta)!}{2^{3/2}\,4^{\beta}\,\beta!\,\pi^{\beta-1/2} }\, |v|^{-\beta-1/2}\,.
	\end{equation}
	For $\beta=2$ we observe by Corollary \ref{Cor:bandlimited} with $M=1$ and $b=2$ the decay $P^{-2.5}$, see Figure \ref{fig:alg_decay}.
\end{example}

\begin{figure}[ht]
	\centering
	\begin{subfigure}[t]{0.48\textwidth}
		\centering
		\begin{tikzpicture}
		\begin{loglogaxis}[FTerrstyle, log basis x=2,
			xlabel={$P$}, ylabel={$M_{P,L}(f)$}, legend pos=north east]
		\addplot[FTest] coordinates {
			(2,5.553604e-01) (4,1.388401e-01) (8,3.471002e-02) (16,8.677506e-03)
			(32,2.169377e-03) (64,5.423442e-04) (128,1.355861e-04) (256,3.389661e-05)
			(512,8.474230e-06) (1024,2.118635e-06)};
		\addplot[FTsim] coordinates {
			(2,1.058614e-01) (4,3.545666e-02) (8,9.894716e-03) (16,2.552464e-03)
			(32,6.433197e-04) (64,1.611598e-04) (128,4.031064e-05) (256,1.007895e-05)
			(512,2.519819e-06) (1024,6.299599e-07)};
		\legend{estimate, $M_{P,L}(f)$}
		\end{loglogaxis}
		\end{tikzpicture}
		\caption{mixed decay: $f(x) = \mathrm{e}^{-2\pi |x|}$ ($L=8$)}
		\label{fig:mixed_decay}
	\end{subfigure}
	\hfill
	\begin{subfigure}[t]{0.48\textwidth}
		\centering
		\begin{tikzpicture}
		\begin{loglogaxis}[FTerrstyle, log basis x=2,
			xlabel={$P$}, ylabel={$M_{P,L}(f)$}, legend pos=north east]
		\addplot[FTest] coordinates {
			(2,2.611784e+00) (4,4.617026e-01) (8,8.161825e-02) (16,1.442820e-02)
			(32,2.550570e-03) (64,4.508814e-04) (128,7.970532e-05) (256,1.409004e-05)};
		\addplot[FTsim] coordinates {
			(2,5.731722e-02) (4,9.908862e-03) (8,1.784438e-03) (16,3.213549e-04)
			(32,5.742238e-05) (64,1.022012e-05) (128,1.812819e-06) (256,3.210081e-07)};
		\legend{estimate, $M_{P,L}(f)$}
		\end{loglogaxis}
		\end{tikzpicture}
		\caption{algebraic decay: $f(x) = (1 - x^2)^{3/2}$ for $x\in [-1,1]$ ($L=2$)}
		\label{fig:alg_decay}
	\end{subfigure}
	\captionsetup{font=small}
	\caption{Estimate and simulated error $M_{P,L}(f)$ for Example \ref{ex:mixeddecay}
		with $L=8$ (left, decay $P^{-2}$) and Example \ref{Ex:alg2.5} with $L=2$
		(right, decay $P^{-2.5}$).}
\end{figure}
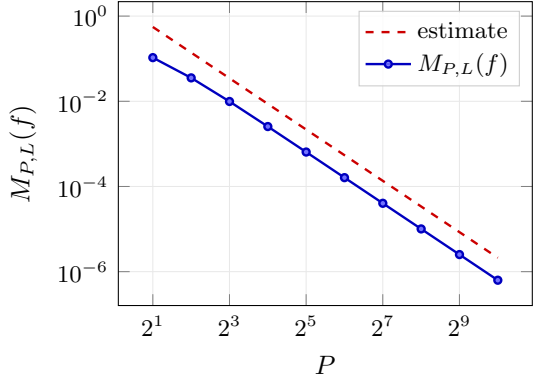
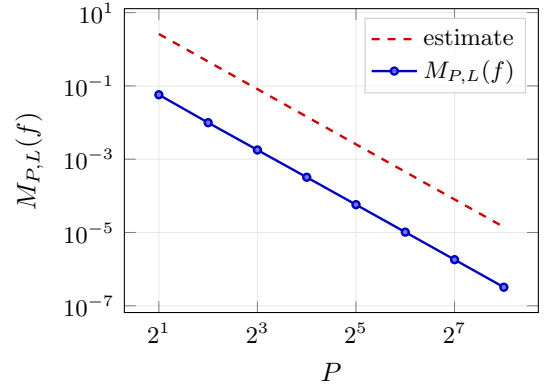

\new{
\begin{example}
\label{Ex:Ldependence}
So far all experiments were carried out for a fixed truncation parameter $L$
and increasing bandwidth $P$, i.e.\ they probe the aliasing term
$L^{-1/2}\,\delta(\hat f,P)$ of \eqref{eq:fsuni}. We now isolate the
\emph{truncation} term $L^{1/2}\,\delta(f,L)$ and keep $P$ fixed while $L$
increases. To this end we use the bandlimited function of
Example~\ref{Ex:M2},
$$
g(x) = \frac{M}{2}\,\Bigl(\mathrm{sinc}\bigl(\tfrac{M}{2}\,x\bigr)\Bigr)^2\,,
\qquad
\hat g(v) = M_2\Bigl(\frac{2v}{M}\Bigr)\,,
\qquad M = 2\,,
$$
for which $\mathrm{supp}\,\hat g \subseteq [-1,\,1]$, hence
$\delta(\hat g, P) = 0$ for all $P \ge 2$, and $g$ has polynomial decay with
$a = 2$ and $c = \sup_x |g(x)|(1+|x|)^2 = 1.267505$. By
Corollary~\ref{Cor:bandlimited} we obtain the estimate
$$
M_{P,L}(g) \ \le\ 2c\,(2^2-1)\,\zeta(2)\,L^{-3/2} = \pi^2 c\,L^{-3/2}\,,
$$
and by Theorem~\ref{Thm:lowerbound} the matching lower bound of the same order
$L^{-3/2}$.

Figure~\ref{fig:Ldep} shows the measured errors for $L = 2^1,\ldots,2^9$ and
$P = 8$. The observed convergence order, computed from consecutive values, is
$-1.457$, $-1.487$, $-1.497$, $-1.499$, $-1.500$, $-1.500$, $-1.500$, $-1.500$,
i.e.\ it converges to the predicted exponent $-\frac32 = \frac12 - a$. Repeating
the experiment with $P = 32$ reproduces the same values to five significant
digits, which confirms that the truncation term of \eqref{eq:fsuni} is indeed
independent of $P$.

The same figure shows the scaled Euclidean error $E_{P,L}^{\ast}(g)$ of
\cite{EhGrKl24}. In this truncation dominated regime both errors have the same
order, and the quotient $E_{P,L}^{\ast}(g)/M_{P,L}^{\ast}(g)$ increases
monotonically from $1.216$ at $L = 2$ to the limit $1.244$, in agreement with
Remark~\ref{Rem:symmetry}. Hence, for this class of functions, our uniform
estimate loses nothing compared with the Euclidean estimate, although it
controls the error at every $v$ and not only at the $LP$ grid points.
\end{example}

% ===========================================================================
%  L-dependence and the quotient E*/M*
% ===========================================================================
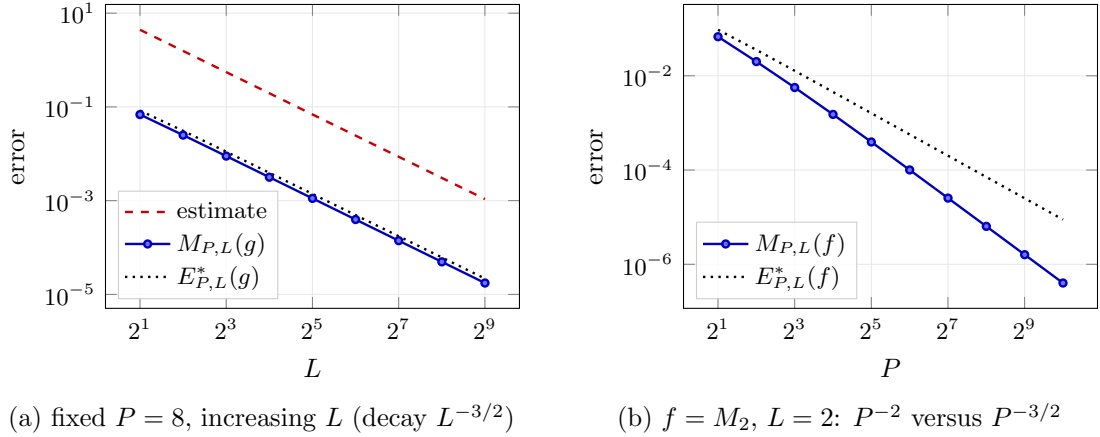
\begin{figure}[ht]
	\centering
	\begin{subfigure}[t]{0.48\textwidth}
		\centering
		\begin{tikzpicture}
		\begin{loglogaxis}[FTerrstyle, log basis x=2,
			xlabel={$L$}, ylabel={error}, legend pos=south west]
		\addplot[FTest] coordinates {
			(2,4.422872e+00) (4,1.563721e+00) (8,5.528589e-01) (16,1.954652e-01)
			(32,6.910737e-02) (64,2.443314e-02) (128,8.638421e-03)
			(256,3.054143e-03) (512,1.079803e-03)};
		\addplot[FTsim] coordinates {
			(2,6.870847e-02) (4,2.502981e-02) (8,8.927725e-03) (16,3.163788e-03)
			(32,1.119230e-03) (64,3.957663e-04) (128,1.399297e-04)
			(256,4.947308e-05) (512,1.749142e-05)};
		\addplot[black, dotted, mark=none] coordinates {
			(2,8.357928e-02) (4,3.088988e-02) (8,1.108188e-02) (16,3.933932e-03)
			(32,1.392307e-03) (64,4.923844e-04) (128,1.740957e-04)
			(256,6.155313e-05) (512,2.176241e-05)};
		\legend{estimate, $M_{P,L}(g)$, $E^{\ast}_{P,L}(g)$}
		\end{loglogaxis}
		\end{tikzpicture}
		\caption{fixed $P = 8$, increasing $L$ (decay $L^{-3/2}$)}
		\label{fig:Ldep}
	\end{subfigure}
	\hfill
	\begin{subfigure}[t]{0.48\textwidth}
		\centering
		\begin{tikzpicture}
		\begin{loglogaxis}[FTerrstyle, log basis x=2,
			xlabel={$P$}, ylabel={error}, legend pos=south west]
		\addplot[FTsim] coordinates {
			(2,6.748202e-02) (4,2.002018e-02) (8,5.637129e-03) (16,1.515243e-03)
			(32,3.939908e-04) (64,1.005330e-04) (128,2.539690e-05)
			(256,6.382778e-06) (512,1.599926e-06) (1024,4.005130e-07)};
		\addplot[black, dotted, mark=none] coordinates {
			(2,9.471527e-02) (4,3.544960e-02) (8,1.277444e-02) (16,4.540515e-03)
			(32,1.607520e-03) (64,5.685408e-04) (128,2.010270e-04)
			(256,7.107530e-05) (512,2.512905e-05) (1024,8.884473e-06)};
		\legend{$M_{P,L}(f)$, $E^{\ast}_{P,L}(f)$}
		\end{loglogaxis}
		\end{tikzpicture}
		\caption{$f = M_2$, $L = 2$: $P^{-2}$ versus $P^{-3/2}$}
		\label{fig:EstarM}
	\end{subfigure}
	\captionsetup{font=small}
	\caption{Left: truncation dominated regime of Example~\ref{Ex:Ldependence};
	both error measures decay like $L^{-3/2}$ and their quotient tends to
	$1.244$. Right: aliasing dominated regime (B-spline $M_2$, $L = 2$); the
	scaled maximum error decays like $P^{-b} = P^{-2}$, the scaled Euclidean
	error only like $P^{-b+1/2} = P^{-3/2}$, so that their quotient grows like
	$\sqrt{LP}$, cf.\ Remark~\ref{Rem:notequivalent}.}
\end{figure}

\section{Computational cost\label{sec:cost}}

The estimates of Section~\ref{sec:samplingpolynomials} answer the question how
accurate the sampling polynomial $s_{P,L}$ is. We now address the complementary
question how expensive it is to evaluate it, and compare the three possibilities
listed in Algorithm~\ref{Alg:FTviaNFFT}:
\begin{itemize}
\item \emph{direct evaluation} of \eqref{eq:samplingpoly} at $K$ nodes, which
requires $\mathcal{O}(LPK)$ arithmetic operations,
\item the \emph{\FFT}, which delivers the $LP$ values $s_{P,L}(k/L)$,
$k \in [LP]$, on the equispaced grid in $\mathcal{O}(LP\log LP)$ operations,
\item the \emph{\NFFT}, which delivers the values $s_{P,L}(v_j)$,
$j = 1,\ldots,K$, at arbitrary nodes in
$\mathcal{O}(LP\log LP + K)$ operations.
\end{itemize}

All computations were performed in double precision on an
Intel\textsuperscript{\textregistered} Core\texttrademark{} i7-6700 CPU at
$3.40$\,GHz, using Python~3.13 with NumPy~2.4 and the NFFT3 library
\cite{nfft3} through its Python interface \texttt{pyNFFT3}~1.0.2
\cite{pyNFFT3}. As a test function we use the two-sided exponential
$f(x) = {\mathrm e}^{-2\pi|x|}$ of Example~\ref{ex:mixeddecay} with $L = 8$, for
which $\hat f(v) = \bigl(\pi(1+v^2)\bigr)^{-1}$ is known in closed form. Timings
are the minimum over five runs after one warm-up run.

Table~\ref{tab:fft} compares the FFT with direct evaluation on the equispaced
grid. Both compute the same quantity, and indeed they agree to machine
precision; the FFT is faster by up to four orders of magnitude. The last column
recalls the corresponding approximation error $M_{P,L}(f)$, which is the
quantity bounded by Theorem~\ref{Thm:mixeddecay}. It is worth noting that the
difference between the two algorithms is purely a matter of speed: the accuracy
of the computed Fourier transform is governed entirely by $M_{P,L}(f)$ and not
by the evaluation scheme.

\begin{table}[ht]
\centering
\captionsetup{font=small}
\caption{Equispaced output nodes $v = k/L$, $k \in [LP]$:
direct evaluation versus FFT, for $f(x) = {\mathrm e}^{-2\pi|x|}$ and $L = 8$.
The column $\|\cdot\|_{\infty}$ contains the maximum deviation between the two
computed results.}
\label{tab:fft}
\small
\begin{tabular}{rrrrrrr}
\hline
$P$ & $LP$ & $t_{\mathrm{direct}}$\,[s] & $t_{\mathrm{FFT}}$\,[s] &
speed-up & $\|\cdot\|_{\infty}$ & $M_{P,L}(f)$\\
\hline
$16$   & $128$   & $8.0\cdot 10^{-4}$ & $4.2\cdot 10^{-5}$ & $19$    & $1.7\cdot10^{-16}$ & $2.55\cdot10^{-3}$\\
$32$   & $256$   & $3.2\cdot 10^{-3}$ & $4.7\cdot 10^{-5}$ & $68$    & $2.5\cdot10^{-16}$ & $6.43\cdot10^{-4}$\\
$64$   & $512$   & $1.2\cdot 10^{-2}$ & $5.8\cdot 10^{-5}$ & $214$   & $2.2\cdot10^{-16}$ & $1.61\cdot10^{-4}$\\
$128$  & $1024$  & $5.1\cdot 10^{-2}$ & $7.5\cdot 10^{-5}$ & $683$   & $3.9\cdot10^{-16}$ & $4.03\cdot10^{-5}$\\
$256$  & $2048$  & $2.1\cdot 10^{-1}$ & $1.1\cdot 10^{-4}$ & $1853$  & $5.6\cdot10^{-16}$ & $1.01\cdot10^{-5}$\\
$512$  & $4096$  & $9.0\cdot 10^{-1}$ & $1.8\cdot 10^{-4}$ & $4895$  & $8.9\cdot10^{-16}$ & $2.52\cdot10^{-6}$\\
$1024$ & $8192$  & $3.2\cdot 10^{0}$  & $3.5\cdot 10^{-4}$ & $9271$  & $1.2\cdot10^{-15}$ & $6.30\cdot10^{-7}$\\
$2048$ & $16384$ & $1.3\cdot 10^{1}$  & $6.9\cdot 10^{-4}$ & $18355$ & $2.5\cdot10^{-15}$ & $1.58\cdot10^{-7}$\\
\hline
\end{tabular}
\end{table}

The situation of real interest, however, is the evaluation at \emph{arbitrary}
nodes, where the FFT is not applicable at all. Table~\ref{tab:nfft} compares
the NFFT with direct evaluation at $K = 10\,000$ nodes drawn uniformly at
random from $[-P/2,\,P/2)$. The NFFT is used with oversampling factor
$\sigma = 2$ and cut-off parameter $m = 8$; the resulting relative deviation
$e_{\mathrm{NFFT}}$ from the directly computed values stays at the level of the
round-off error, while the speed-up reaches a factor of $4262$ at
$P = 4096$. The column $\max_j |\hat f(v_j) - s_{P,L}(v_j)|$ shows the error
that is actually committed with respect to the true Fourier transform. It
agrees with $\sqrt{L}\,M_{P,L}(f)$ to three digits in every row, i.e.\ the
random nodes do attain the uniform bound. This illustrates once more that it is
the \emph{uniform} error \eqref{eq:uniform}, and not its discrete version
\eqref{eq:discreteuniform}, which governs the accuracy of the NFFT-based
computation.

\begin{table}[ht]
\centering
\captionsetup{font=small}
\caption{$K = 10\,000$ random nonequispaced nodes in $[-P/2,\,P/2)$: direct
evaluation versus NFFT ($\sigma = 2$, $m = 8$), for
$f(x) = {\mathrm e}^{-2\pi|x|}$ and $L = 8$. Here $e_{\mathrm{NFFT}}$ is the
relative maximum deviation of the NFFT from the direct evaluation.}
\label{tab:nfft}
\footnotesize
\begin{tabular}{@{}rrrrrrr@{}}
\hline
$P$ & $LP$ & $t_{\mathrm{direct}}$\,[s] & $t_{\mathrm{NFFT}}$\,[s] &
speed-up & $e_{\mathrm{NFFT}}$ & $\max_j|\hat f(v_j)-s_{P,L}(v_j)|$\\
\hline
$16$   & $128$   & $7.2\cdot10^{-2}$ & $1.50\cdot10^{-3}$ & $48$   & $3.8\cdot10^{-15}$ & $7.22\cdot10^{-3}$\\
$32$   & $256$   & $1.3\cdot10^{-1}$ & $1.52\cdot10^{-3}$ & $86$   & $3.7\cdot10^{-15}$ & $1.82\cdot10^{-3}$\\
$64$   & $512$   & $2.6\cdot10^{-1}$ & $1.56\cdot10^{-3}$ & $167$  & $4.2\cdot10^{-15}$ & $4.56\cdot10^{-4}$\\
$128$  & $1024$  & $5.1\cdot10^{-1}$ & $1.64\cdot10^{-3}$ & $310$  & $3.8\cdot10^{-15}$ & $1.14\cdot10^{-4}$\\
$256$  & $2048$  & $1.0\cdot10^{0}$  & $1.70\cdot10^{-3}$ & $586$  & $4.2\cdot10^{-15}$ & $2.85\cdot10^{-5}$\\
$512$  & $4096$  & $2.0\cdot10^{0}$  & $1.80\cdot10^{-3}$ & $1121$ & $3.8\cdot10^{-15}$ & $7.13\cdot10^{-6}$\\
$1024$ & $8192$  & $4.0\cdot10^{0}$  & $2.90\cdot10^{-3}$ & $1363$ & $4.5\cdot10^{-15}$ & $1.78\cdot10^{-6}$\\
$2048$ & $16384$ & $7.7\cdot10^{0}$  & $2.56\cdot10^{-3}$ & $3006$ & $9.4\cdot10^{-15}$ & $4.45\cdot10^{-7}$\\
$4096$ & $32768$ & $1.6\cdot10^{1}$  & $3.67\cdot10^{-3}$ & $4262$ & $2.4\cdot10^{-14}$ & $1.11\cdot10^{-7}$\\
\hline
\end{tabular}
\end{table}

Finally, Table~\ref{tab:cutoff} documents step~4 of
Algorithm~\ref{Alg:FTviaNFFT}, i.e.\ the choice of the NFFT cut-off parameter
$m$. As predicted by \cite{PoTa21a,PoTa21b}, the NFFT error decays
exponentially in $m$, while the runtime grows only mildly. For $P = 256$ and
$L = 8$ the approximation error is
$\sqrt{L}\,M_{P,L}(f) = 2.85\cdot 10^{-5}$, so that already $m = 3$ makes the
NFFT error negligible; the frequently used default $m = 8$ reduces it to the
level of round-off. In other words, the accuracy of the whole computation is
determined by the parameters $L$ and $P$ through
Theorem~\ref{Thm:samplingpolynomials}, and the NFFT contributes no relevant
error at negligible additional cost.

\begin{table}[ht]
\centering
\captionsetup{font=small}
\caption{Influence of the NFFT cut-off parameter $m$ for
$f(x)={\mathrm e}^{-2\pi|x|}$, $P = 256$, $L = 8$, $K = 10\,000$ and
$\sigma = 2$. For comparison, the approximation error of the sampling
polynomial is $\sqrt{L}\,M_{P,L}(f) = 2.85\cdot 10^{-5}$.}
\label{tab:cutoff}
\footnotesize
\begin{tabular}{@{}lrrrrrrr@{}}
\hline
$m$ & $2$ & $3$ & $4$ & $5$ & $6$ & $7$ & $8$\\
\hline
$t_{\mathrm{NFFT}}$\,[ms] & $1.6$ & $2.1$ & $1.6$ & $1.6$ & $1.6$ & $1.6$
& $1.7$\\
$e_{\mathrm{NFFT}}$ & $7.7$e$-5$ & $8.7$e$-7$ & $2.3$e$-8$ & $1.5$e$-10$
& $1.1$e$-12$ & $1.8$e$-14$ & $4.2$e$-15$\\
\hline
\end{tabular}
\end{table}

\section{Conclusion\label{sec:conclusion}}

We have studied how well the continuous Fourier transform $\hat f$ of an
integrable function $f$ can be approximated by the $P$-periodic trigonometric
sampling polynomial $s_{P,L}$ of degree $LP/2$, measured in the uniform norm on
the whole frequency interval $[-P/2,\,P/2]$.

The central result is Theorem~\ref{Thm:samplingpolynomials}, which bounds the
scaled maximum approximation error $M_{P,L}(f)$ by the linear combination
$L^{-1/2}\,\delta(\hat f,P) + L^{1/2}\,\delta(f,L)$ of the decay rates
\eqref{eq:decayrate} of $\hat f$ and $f$. Combined with the explicit estimates
of the decay rates in Lemma~\ref{Lemma:polydecay} and
Lemma~\ref{Lemma:expodecay}, this yields fully explicit error bounds under
polynomial, exponential and mixed decay assumptions
(Theorems~\ref{Thm:polydecay}, \ref{Thm:expodecay} and
\ref{Thm:mixeddecay}), in which every constant is determined by the decay
parameters of $f$ and $\hat f$. Theorem~\ref{Thm:lowerbound} shows that the
resulting rate in the truncation parameter $L$ is sharp.

These estimates have two consequences of practical importance. First, they can
be used a~priori: given a target accuracy, the parameters $L$ and $P$ can be
determined before the computation is started, which is what
Algorithm~\ref{Alg:FTviaNFFT} does. Second, since the bound is uniform in $v$,
it remains valid when $s_{P,L}$ is evaluated at arbitrary nodes by an NFFT. The
NFFT thereby becomes a numerical method for the Fourier transform itself and
not merely a fast evaluation scheme for trigonometric polynomials. 

Our results also clarify the relation to the work of M.~Ehler, K.~Gr\"ochenig
and A.~Klotz \cite{EhGrKl24}. The scaled maximum error and the scaled Euclidean
error are not equivalent; their quotient ranges over $[1,\,\sqrt{LP}\,]$, and
both extremes occur (Remark~\ref{Rem:notequivalent}). In the truncation
dominated regime the two error measures agree up to a constant, whereas in the
aliasing dominated regime the uniform estimate derived here is better by the
factor $\sqrt{LP}$; Remark~\ref{Rem:symmetry} traces this back to the different
oscillatory behaviour of the aliasing and the truncation remainder and explains
why no symmetric bound can hold for the uniform error.
}

\section*{Acknowledgements}

The authors would like to thank M.~Ehler, K.~Gr\"{o}chenig, and H.G.~Feichtinger for the fruitful discussions.
\new{They are further indebted to the  reviewers for their careful
reading of the manuscript and for their constructive comments, which lead in
particular to the lower bound in Theorem~\ref{Thm:lowerbound}, to the
discussion of the time--frequency asymmetry in
Remark~\ref{Rem:symmetry} and Remark~\ref{Rem:notequivalent}, and to the new
Sections~\ref{sec:literature} and~\ref{sec:cost}.}

\bibliographystyle{abbrv}

\end{document}